\documentclass{amsart}
\usepackage{amssymb}
\usepackage{latexsym}
\usepackage{amsmath}
\usepackage{amscd}

\newcommand{\dom}{{\mathrm{dom\,}}}
\newcommand{\ran}{{\mathrm{ran\,}}}

\newcommand{\tr}{{\mathrm{tr}}}

\def\RE{{\rm Re\,}}
\def\IM{{\rm Im\,}}

\newtheorem{theorem}{Theorem}[section]
\newtheorem{proposition}[theorem]{Proposition}
\newtheorem{corollary}[theorem]{Corollary}
\newtheorem{lemma}[theorem]{Lemma}

\theoremstyle{definition}

\newtheorem{definition}[theorem]{Definition}
\newtheorem{remark}[theorem]{Remark}

\def\sS{{\mathfrak S}}

      \def\dC{{\mathbb C}}

   \def\dN{{\mathbb N}}   
      \def\dR{{\mathbb R}}

      \def\cC{{\mathcal C}}
\def\cD{{\mathcal D}}      
\def\cG{{\mathcal G}}   \def\cH{{\mathcal H}}   
      \def\cL{{\mathcal L}}
\def\cM{{\mathcal M}}   \def\cN{{\mathcal N}}

\def\sp{{\text{\rm sp\,}}}

\def\cspan{{\rm \overline{span}\, }}

\begin{document}

\title[$Q$-functions and elliptic operators]
{Generalized $Q$-functions and Dirichlet-to-Neumann maps for elliptic differential operators}
\author{Daniel Alpay$^1$ and Jussi Behrndt$^2$}

\thanks{\noindent $^1$ Earl Katz Family Chair in algebraic system theory; $\quad ^2$ Jussi Behrndt gratefully acknowledges support from the Institute for Advanced Studies in Mathematics at Ben-Gurion University of the Negev}

\address{Department of Mathematics \\
Ben-Gurion University of the Negev \\
P.O. Box 653\\
84105 Beer-Sheva\\
Israel}
\email{dany@math.bgu.ac.il}

\address{Institut f\"ur Mathematik, MA 6-4 \\
Technische Universit\"at Berlin \\
Strasse des 17. Juni 136\\
10623 Berlin \\
Deutschland}
\email{behrndt@math.tu-berlin.de}

\maketitle

\begin{abstract}
The classical concept of $Q$-functions associated to symmetric and selfadjoint operators due
to M.G.~Krein and H.~Langer is extended
in such a way that the Dirichlet-to-Neumann map in the theory of elliptic differential
equations can be interpreted as a generalized $Q$-function. For couplings of uniformly elliptic second order
differential expression on bounded and unbounded domains explicit Krein type formulas  
for the difference of the resolvents and trace formulas in an $H^2$-framework are obtained.
\end{abstract}

\section{Introduction}

The notion of a $Q$-function associated to a pair $\{S,A\}$ consisting of a symmetric operator 
$S$ and a selfadjoint extension $A$ of $S$ in a Hilbert or Pontryagin space was introduced by 
M.G.~Krein and H.~Langer in \cite{KL73,KL77}. A $Q$-function contains the spectral information of the selfadjoint extensions
of the underlying symmetric operator and therefore these functions play a very
important role in the spectral and perturbation theory of selfadjoint operators. $Q$-functions appear also naturally 
in the description of the resolvents of the selfadjoint extensions of a symmetric operator with the help of Krein's
formula and they can be used to construct functional models for selfadjoint operators. In the theory of 
boundary triplets associated to symmetric operators  
$Q$-functions can be interpreted as so-called Weyl functions, cf. \cite{BGP08,DHMS06,DM91,DM95,GG91}.
A prominent example for a $Q$-function is the classical Titchmarsh-Weyl coefficient
in the theory of singular Sturm-Liouville operators.

The main objective of this paper is to extend the concept of $Q$-functions in such a way that
the Dirichlet-to-Neumann map in the theory of elliptic differential equations can be identified
as a generalized $Q$-function. In the abstract part of the paper we introduce the notion of 
generalized $Q$-functions
and we show that these functions have similar properties as classical $Q$-functions. 
Besides a symmetric operator $S$ and a selfadjoint extension $A$
also an operator $T$ whose closure coincides with $S^*$ is used. Some of the ideas here parallel \cite{BL07}, where 
a more abstract approach with isometric and unitary relations in Krein spaces was used. 
The main result in the abstract part is Theorem~\ref{qthmgen1} which states that an operator function is a
generalized $Q$-function if and only if it coincides    
up to a possibly unbounded constant on a dense subspace with the restriction of a
Nevanlinna function with an invertible imaginary part and a certain asymptotic behaviour. 

Section~\ref{ellops} and Section~\ref{cellops} deal with second order
elliptic operators on bounded and unbounded domains, and with the coupling of such operators. 
Suppose first that the domain $\Omega\subset\dR^n$, $n>1$, is
bounded with a smooth boundary $\partial\Omega$. Let $A_D$ and $A_N$ be the selfadjoint 
realizations of an formally symmetric uniformly elliptic differential expression 
\begin{equation}\label{cl1}
\cL=-\sum_{j,k=1}^n  \frac{\partial}{\partial x_j} \,a_{jk} \frac{\partial }{\partial x_k}+ a
\end{equation}
in $L^2(\Omega)$ defined on $H^2(\Omega)$ and 
subject to Dirichlet and Neumann boundary conditions, respectively. If $T$ denotes the realization of $\cL$
on $H^2(\Omega)$, then the closure of $T$ in $L^2(\Omega)$ coincides with the maximal operator associated to $\cL$ in 
$L^2(\Omega)$, and $A_D$ and $A_N$ are both selfadjoint restrictions of $T$.
For a function $f\in H^2(\Omega)$ denote the trace and the trace of the conormal derivative by $f|_{\partial\Omega}$
and $\tfrac{\partial f}{\partial\nu}|_{\partial\Omega}$, respectively. Then for each $\lambda\in\rho(A_D)$ the
Dirichlet-to-Neumann map 
\begin{equation}\label{dnmapintro}
Q(\lambda)(f_\lambda|_{\partial\Omega}):= -\frac{\partial f_\lambda}{\partial\nu}\Bigl|_{\partial\Omega},\qquad
\text{where}\quad T f_\lambda=\lambda f_\lambda,
\end{equation}
is well-defined and will be regarded as an operator in $L^2(\partial\Omega)$ defined on 
$H^{3/2}(\partial\Omega)$ with values in $H^{1/2}(\partial\Omega)$. The minus sign in \eqref{dnmapintro}
is used for technical reasons. It turns out that the operator 
function $\lambda\mapsto Q(\lambda)$ is a generalized $Q$-function in the sense of Definition~\ref{defq} 
and an explicit variant of Krein's formula for the resolvents of $A_D$ and $A_N$ is obtained in Theorem~\ref{bigthm1}, 
see also \cite{BL07,BGW09,GM08,GM08-2,P08,PR09,Post07} for more general problems. In particular, 
in the case $n=2$ the difference of these resolvents is a trace class operator and we obtain the trace formula
\begin{equation}\label{traceformi}
\tr\bigl((A_D-\lambda)^{-1}-(A_N-\lambda)^{-1}\bigr)
=\tr\left(\overline{Q(\lambda)^{-1}}\,\frac{d}{d\lambda}\,\widetilde Q(\lambda)\right)
\end{equation}
for $\lambda\in\rho(A_D)\cap\rho(A_N)$. Here $\overline{Q(\lambda)^{-1}}$ is the closure of $Q(\lambda)^{-1}$ 
in $L^2(\partial\Omega)$ and 
$\widetilde Q$ is a Nevanlinna function which differs from the Dirichlet-to-Neumann map by a symmetric
constant. Trace formulas for canonical differential expressions and in more abstract situations for 
the finite-dimensional case can be found in, e.g., \cite{AG03,AG05,BMN08}.

In Section~\ref{cellops} we consider a so-called coupling of elliptic operators. Such couplings are of great interest
in problems of mathematical physics, e.g., in the description of quantum networks; for more details and further references 
we refer the reader to the recent works \cite{EK04,EK04-2,MPP07,MPR07,P07}. 
Suppose that $\dR^n$, $n>1$, 
is decomposed in a bounded domain $\Omega$ with smooth boundary $\cC$ and the unbounded domain 
$\Omega^\prime=\dR^n\backslash\overline\Omega$. The orthogonal sum of the selfadjoint Dirichlet operators 
$A_D$ and $A_D^\prime$ associated to $\cL$ in $L^2(\Omega)$ and $L^2(\Omega^\prime)$, respectively, is regarded as a 
selfadjoint diagonal block
operator matrix in $L^2(\dR^n)$. 
The resolvent of $A_D\oplus A_D^\prime$ is then compared with the resolvent of
the usual selfadjoint realization $\widetilde A$ of $\cL$ in $L^2(\dR^n)$ defined on $H^2(\dR^n)$.
In order to express this difference in the Krein type formula 
\begin{equation}\label{kreinintro}
\bigl((A_D\oplus A_D^\prime)-\lambda\bigr)^{-1}-(\widetilde A-\lambda)^{-1}=
\Gamma(\lambda) Q(\lambda)^{-1}\Gamma(\bar\lambda)^*
\end{equation}
with a generalized $Q$-function an analogon 
of the Dirichlet-to-Neumann map is constructed which measures the jump of the conormal derivative 
of $L^2(\Omega)$ and $L^2(\Omega^\prime)$-solutions of $\cL u=\lambda u$ on the boundary $\cC$, see \eqref{qcoup}. 
The operator 
$\Gamma(\lambda):L^2(\cC)\rightarrow L^2(\dR^n)$
in \eqref{kreinintro} is closely connected with the generalized $Q$-function and is here identified
with a Poisson-type operator solving a certain Dirichlet problem.
As a consequence of the representation \eqref{kreinintro} we also obtain a trace formula of the 
type \eqref{traceformi} in the coupled case.

\section{Generalized $Q$-functions}\label{genq}

In this section we introduce the notion of generalized $Q$-functions associated to a symmetric
operators in Hilbert spaces. The class of generalized $Q$-functions is characterized in Theorem~\ref{qthmgen1}, where 
it turns out that generalized $Q$-functions are closely connected with operator-valued Nevanlinna or
Riesz-Herglotz functions. We also note in advance that for the case of finite deficiency indices of the underlying
symmetric operator the concept of generalized $Q$-functions coincides with the classical notion of
(ordinary) $Q$-functions studied by M.G.~Krein and H.~Langer in \cite{KL73,KL77}, see also \cite{K47,K49}.

Let $\cH$ be a separable Hilbert space and let $S$ be a densely defined closed symmetric operator 
with equal (in general infinite) deficiency indices 
$$n_\pm(S)=\dim\ker(S^*\mp i)\leq \infty$$ in $\cH$. 
It is well known that under this assumption $S$ admits selfadjoint extensions in $\cH$.
In the following let $A$ be a fixed selfadjoint extension of $S$ in $\cH$, so that,
$S\subset A=A^*\subset S^*$. Furthermore, let $T$ be a 
linear operator in $\cH$ 
such that $A\subset T\subset S^*$ and $\overline T=S^*$ 
holds, i.e., the domain $\dom T$ of $T$ is a core of $\dom S^*$ (see \cite{K76}), 
$\dom T$ contains $\dom A$ and $Af=Tf$ holds for all $f\in\dom A$.

For $\lambda\in\dC$ belonging to the resolvent set $\rho(A)$ of the selfadjoint operator $A$ define the defect spaces 
$\cN_\lambda(T)=\ker(T-\lambda)$ and 
$\cN_\lambda(S^*)=\ker(S^*-\lambda)$. Then the decompositions
\begin{equation}\label{decoall}
\dom S^*=\dom A\,\dot +\,\cN_\lambda(S^*)\quad\text{and}\quad
\dom T=\dom A\,\dot +\,\cN_\lambda(T)
\end{equation}
hold for all $\lambda\in\rho(A)$ and the closure $\overline{\cN_\lambda(T)}$ 
of $\cN_\lambda(T)$ in $\cH$ coincides with $\cN_\lambda(S^*)$. Recall that the symmetric operator $S$ is said to be 
{\it simple} if there exists no nontrivial subspace $\cD$ in $\dom S$ such that $S$ restricted to $\cD$
is a selfadjoint operator in the Hilbert space $\overline\cD$. It is important to note that $S$ is simple if and only if
\begin{equation}\label{cspan}
\cH=\cspan\bigl\{\cN_\lambda(S^*):\lambda\in\dC\backslash\dR\bigr\}
\end{equation}
holds, cf. \cite{K49}. Here $\cspan$ denotes the closed linear span.
As $\overline{\cN_\lambda(T)}=\cN_\lambda(S^*)$ it is clear that the right hand side in \eqref{cspan} coincides with 
\begin{equation*}
\cspan\bigl\{\cN_\lambda(T):\lambda\in\dC\backslash\dR\bigr\}.
\end{equation*}

Fix some $\lambda_0\in\rho(A)$, let $\cG$ be a Hilbert space with the same dimension as $\cN_{\lambda_0}(T)$ and let
$\Gamma_{\lambda_0}$ be a densely defined bounded operator from $\cG$ into $\cH$ such that
$\ran\Gamma_{\lambda_0}=\cN_{\lambda_0}(T)$ and $\ker\Gamma_{\lambda_0}=\{0\}$ holds. 
The domain $\dom\Gamma_{\lambda_0}$ of $\Gamma_{\lambda_0}$ will be denoted by $\cG_0$. Observe that the closure 
$\overline\Gamma_{\lambda_0}$
of the operator $\Gamma_{\lambda_0}$ is the bounded extension of $\Gamma_{\lambda_0}$ which is defined on 
$\overline\cG_0=\cG$. We write $\overline\Gamma_{\lambda_0}\in\cL(\cG,\cH)$, where $\cL(\cG,\cH)$ is the space
of bounded linear operators defined on $\cG$ with values in $\cH$.

\begin{lemma}\label{gamlem}
The operator function $\lambda\mapsto\Gamma(\lambda):=(I+(\lambda-\lambda_0)(A-\lambda)^{-1})\Gamma_{\lambda_0}$
satisfies $\Gamma(\lambda_0)=\Gamma_{\lambda_0}$,
\begin{equation*}
\Gamma(\lambda)=\bigl(I+(\lambda-\mu)(A-\lambda)^{-1}\bigr)\Gamma(\mu),\qquad \lambda,\mu\in\rho(A),
\end{equation*}
and $\Gamma(\lambda)$ is a bounded operator from $\cG$ into $\cH$ which maps $\dom \Gamma(\lambda)=\cG_0$ 
bijectively onto $\cN_\lambda(T)$ for all $\lambda\in\rho(A)$. 
Moreover, $\lambda\mapsto\Gamma(\lambda)g$ is holomorphic
on $\rho(A)$ for every $g\in\cG_0$.
\end{lemma}

\begin{proof}
Let us show that $\ran\Gamma(\lambda)=\cN_\lambda(T)$ is true. The other assertions in the lemma are obvious or 
follow from a straightforward calculation.  Since $T$ is an extension of $A$ we have $(T-\lambda)(A-\lambda)^{-1}=I$ for $\lambda\in\rho(A)$
and therefore
\begin{equation*}
(T-\lambda)\Gamma(\lambda)h=(T-\lambda)\bigl(I+(\lambda-\lambda_0)(A-\lambda)^{-1}\bigr)\Gamma_{\lambda_0} h
=(T-\lambda_0)\Gamma_{\lambda_0} h=0
\end{equation*}
shows that $\ran\Gamma(\lambda)\subset\cN_\lambda(T)$ holds. Now let $f_\lambda\in\cN_\lambda(T)$. Then it
follows as above that
\begin{equation*}
f_{\lambda_0}:=\bigl(I+(\lambda_0-\lambda)(A-\lambda_0)^{-1}\bigr) f_\lambda
\end{equation*}
is an element in $\cN_{\lambda_0}(T)$ and hence there exists $h\in\cG_0$ such that $f_{\lambda_0}=\Gamma_{\lambda_0}h$.
Now a simple calculation shows $f_\lambda=\Gamma(\lambda)h$, thus $\ran\Gamma(\lambda)=\cN_\lambda(T)$. 
\end{proof}

In the following definition the concept of generalized $Q$-functions is introduced.

\begin{definition}\label{defq}
Let $S$, $A$, $T$, and $\Gamma(\cdot)$ be as above. An operator function $Q$ defined on $\rho(A)$ 
whose values $Q(\lambda)$ are linear operators in $\cG$ with $\dom Q(\lambda)=\cG_0$ for all $\lambda\in\rho(A)$ 
is said to be a {\em generalized $Q$-function} of the triple $\{S,A,T\}$ if 
\begin{equation}\label{q}
Q(\lambda)-Q(\mu)^*=(\lambda-\bar\mu)\Gamma(\mu)^*\Gamma(\lambda)
\end{equation}
holds for all $\lambda,\mu\in\rho(A)$. 
If, in addition, $\cG_0=\cG$ and $T=S^*$, then $Q$ is called
an {\em ordinary $Q$-function} of $\{S,A\}$.
\end{definition}

We note that the values $Q(\lambda)$, $\lambda\in\rho(A)$, of a generalized $Q$-function can 
be unbounded non-closed operators. The adjoint
$Q(\mu)^*$ in \eqref{q} is well defined since $\dom Q(\mu)$ is dense in $\cG$ and by setting $\lambda=\bar\mu$
in \eqref{q} it follows $Q(\mu)\subset Q(\bar\mu)^*$. Hence the identity \eqref{q} holds on $\cG_0$,
the operators $Q(\lambda)$ are closable in $\cG$ and 
symmetric for $\lambda\in\rho(A)\cap\dR$. 
The real and imaginary parts of the operators $Q(\lambda)$ are defined as usual:
\begin{equation*}
\RE Q(\lambda)=\frac{1}{2}\bigl(Q(\lambda)+Q(\lambda)^*\bigr)\quad\text{and}\quad
\IM Q(\lambda)=\frac{1}{2i}\bigl(Q(\lambda)-Q(\lambda)^*\bigr).
\end{equation*}
Since $(\RE Q(\lambda)h,h)$ and $(\IM Q(\lambda)h,h)$ are real for all $h\in\cG_0$ the operators 
$\RE Q(\lambda)$ and $\IM Q(\lambda)$ are symmetric.

\begin{remark}
We note that the concept of generalized $Q$-functions is closely connected with the theory of boundary triplets
and associated Weyl functions. The Weyl function of an ordinary or generalized 
boundary triplet (see \cite{BGP08,DM91,DM95,GG91}) is also a generalized $Q$-function, but the converse is not
true. 
The class of generalized $Q$-functions
studied here coincides with the class of Weyl functions of so-called quasi boundary triplets 
introduced in \cite{BL07}. Furthermore, we note that generalized $Q$-functions are no subclass of the Weyl families
associated to boundary relations, see \cite{DHMS06} and Theorem~\ref{qthmgen1}. 
\end{remark}

The concept of generalized $Q$-functions differs from the classical notion of ordinary $Q$-functions only in the
case $n_\pm(S)=\infty$.

\begin{proposition}
Let $Q$ be a generalized $Q$-function of the triple $\{S,A,T\}$ and assume, in addition, that the deficiency
indices $n_\pm(S)$ are finite. Then $T=S^*$ and $Q$ is an ordinary $Q$-function of the pair $\{S,A\}$.
\end{proposition}

\begin{proof}
If the deficiency indices of the closed operator $S$ are finite, then
$T$ is a finite dimensional extension of $S$ and hence also $T$ is closed. Therefore 
$T=\overline T=S^*$. Moreover, in this case
also $\dim\cG=\dim\cN_{\lambda_0}(T)$ is finite and hence $\cG_0=\dom\Gamma(\lambda)=\dom Q(\lambda)=\cG$, 
$\lambda\in\dC\backslash\dR$.
\end{proof}

The representation of a generalized $Q$-function with the help of the resolvent of $A$ in the next proposition is formally
the same as for ordinary $Q$-functions, see \cite{KL73,KL77,LT77}.

\begin{proposition}\label{formq}
Let $Q$ be a generalized $Q$-function of the triple $\{S,A,T\}$ and let $\lambda_0\in\rho(A)$. Then
$Q$ can be written as the sum of the possibly unbounded operator $\RE Q(\lambda_0)$ and a
bounded holomorphic operator function,
\begin{equation}\label{qa}
Q(\lambda)=\RE Q(\lambda_0)+\Gamma_{\lambda_0}^*\bigl((\lambda-\RE\lambda_0)+(\lambda-\lambda_0)(\lambda-\bar\lambda_0)
(A-\lambda)^{-1}\bigr)\Gamma_{\lambda_0},
\end{equation}
and, in particular, any two generalized $Q$-functions of $\{S,A\}$ differ by a constant.
\end{proposition}

\begin{proof}
Let $h\in\cG$ and set $\mu=\lambda_0$ in \eqref{q}. Making use of the definition of $\Gamma(\lambda)$ in 
Lemma~\ref{gamlem} we obtain
\begin{equation*}
Q(\lambda)h=Q(\lambda_0)^*h+(\lambda-\bar\lambda_0)\Gamma_{\lambda_0}^*
\bigl(I+(\lambda-\lambda_0)(A-\lambda)^{-1}\bigr)\Gamma_{\lambda_0}h.
\end{equation*}
As $Q(\lambda_0)h-Q(\lambda_0)^*h=(\lambda_0-\bar\lambda_0)\Gamma_{\lambda_0}^*\Gamma_{\lambda_0}h$
we see that the above formula can be rewritten as
\begin{equation*}
Q(\lambda)h=Q(\lambda_0)h+(\lambda-\lambda_0)\Gamma_{\lambda_0}^*\Gamma_{\lambda_0}h+ \Gamma_{\lambda_0}^*
(\lambda-\lambda_0)(\lambda-\bar\lambda_0)(A-\lambda)^{-1}\Gamma_{\lambda_0}h.
\end{equation*}
The representation \eqref{qa} follows by inserting 
$Q(\lambda_0)h=\RE Q(\lambda_0)h+i \IM Q(\lambda_0)h$ and $\IM Q(\lambda_0)h=
\IM\lambda_0 \Gamma_{\lambda_0}^*\Gamma_{\lambda_0}h$ into this expression.
\end{proof}

Generalized $Q$-functions are closely connected with the class of Nevanlinna functions, cf. Theorem~\ref{qthmgen1} below.
Let $\cL(\cG)$ be the space of everywhere defined bounded linear operators in $\cG$.
Recall that an $\cL(\cG)$-valued operator function $\widetilde Q$ which is holomorphic on $\dC\backslash\dR$
and satisfies
\begin{equation}\label{imqpos}
\frac{\IM \widetilde Q(\lambda)}{\IM \lambda} \geq 0\qquad\text{and}\qquad \widetilde Q(\bar\lambda)=\widetilde Q(\lambda)^*
\end{equation}
for $\lambda\in\dC\backslash\dR$ is said to be an $\cL(\cG)$-valued {\it Nevanlinna function}. We note that $\widetilde Q$
is an $\cL(\cG)$-valued Nevanlinna function if and only if $\widetilde Q$ admits  an integral representation of the form
\begin{equation}\label{intrepq}
\widetilde Q(\lambda)=\alpha+\lambda \beta+\int_\dR\left(\frac{1}{t-\lambda}-\frac{t}{1+t^2}\right)d\Sigma(t),
\qquad\lambda\in\dC\backslash\dR,
\end{equation}
where $\alpha=\alpha^*\in\cL(\cG)$, $0\leq\beta=\beta^*\in\cL(\cG)$ and $t\mapsto\Sigma(t)\in\cL(\cG)$
is a selfadjoint nondecreasing $\cL(\cG)$-valued function on $\dR$ 
such that 
\begin{equation*}
\int_\dR \frac{1}{1+t^2}\,d\Sigma(t)\in\cL(\cG). 
\end{equation*}
It is well known that Nevanlinna functions can be represented with the help of selfadjoint operators or relations in Hilbert spaces
in a very similar form as in \eqref{qa}. Such operator and functional models for Nevanlinna functions can be found 
in, e.g., \cite{ABDS90,BHS08,B78,BDS93,DM95,GT00,HSW98,LT77,MM03}.

In the next theorem we characterize the class of generalized $Q$-functions. Roughly speaking, it turns out
that up to a symmetric constant a generalized $Q$-function is a restrictions of an 
$\cL(\cG)$-valued Nevanlinna function $\widetilde Q$ with invertible
imaginary part on $\dom Q(\lambda)$ and $\widetilde Q$ satisfies certain limit properties at $\infty$.

\begin{theorem}\label{qthmgen1}
Let $\cG_0$ be a dense subspace of $\cG$, $\lambda_0\in\dC\backslash\dR$, and let 
$Q$ be a function defined on $\dC\backslash\dR$ whose values $Q(\lambda)$ are linear operators in $\cG$ 
with $\dom Q(\lambda)=\cG_0$, $\lambda\in\dC\backslash\dR$. Then the following is equivalent:
\begin{enumerate}
\item [{\rm (i)}] $Q$ is a generalized
$Q$-function of a triple $\{S,A,T\}$, where $S$ is a simple symmetric operator in some separable Hilbert space $\cH$, 
$A$ is a selfadjoint extension of $S$ in $\cH$ and $A\subset T\subset S^*$ with $\overline T=S^*$;
\item [{\rm (ii)}] There exists an unique $\cL(\cG)$-valued Nevanlinna function $\widetilde Q$
with the properties {\rm ($\alpha$), ($\beta$)} and {\rm ($\gamma$)}:
\begin{enumerate}
\item [{\rm ($\alpha$)}] The relations
\begin{equation*}
Q(\lambda)h-\RE Q(\lambda_0)h=\widetilde Q(\lambda)h
\end{equation*}
and
\begin{equation*}
Q(\lambda)^*h-\RE Q(\lambda_0)h=\widetilde Q(\lambda)^*h
\end{equation*}
hold for all $h\in\cG_0$ and $\lambda\in\dC\backslash\dR$;
\item [{\rm ($\beta$)}] $\IM \widetilde Q(\lambda)h=0$ for some $h\in\cG_0$ and $\lambda\in\dC\backslash\dR$ 
implies $h=0$;
\item [{\rm ($\gamma$)}] The conditions 
\begin{equation*}\label{conds}
\lim_{\eta\rightarrow +\infty} \frac{1}{\eta}(\widetilde Q(i\eta)k,k)=0\quad \text{and}\quad
\lim_{\eta\rightarrow +\infty} \eta\,\IM (\widetilde Q(i\eta) k,k)=\infty
\end{equation*}
are valid for all $k\in\cG$, $k\not=0$.
\end{enumerate}
\end{enumerate}
\end{theorem}

\begin{proof}
We start by showing that (i) implies (ii). For this, let $Q$ be a generalized $Q$-function of the triple $\{S,A,T\}$
and suppose that $S$ is simple. Let $\Gamma_{\lambda_0}$ be a bounded operator defined on $\dom Q(\lambda)=\cG_0$
such that $\ran\Gamma_{\lambda_0}=\cN_{\lambda_0}(T)$ and $\ker\Gamma_{\lambda_0}=\{0\}$. 
According to Proposition~\ref{formq} for each $\lambda\in\dC\backslash\dR$
\begin{equation*}
Q(\lambda)-\RE Q(\lambda_0)=\Gamma_{\lambda_0}^*\bigl((\lambda-\RE\lambda_0)+(\lambda-\lambda_0)(\lambda-\bar\lambda_0)
(A-\lambda)^{-1}\bigr)\Gamma_{\lambda_0}
\end{equation*}
is a bounded operator in $\cG$ defined on the dense subspace $\cG_0$ and hence admits a unique bounded extension onto $\cG$ which is given by
\begin{equation}\label{qtilde}
\widetilde Q(\lambda):=\Gamma_{\lambda_0}^*\bigl((\lambda-\RE\lambda_0)+(\lambda-\lambda_0)(\lambda-\bar\lambda_0)
(A-\lambda)^{-1}\bigr)\overline\Gamma_{\lambda_0},
\end{equation}
where $\overline\Gamma_{\lambda_0}\in\cL(\cG,\cH)$ is the closure of $\Gamma_{\lambda_0}$. Obviously we have
\begin{equation*}
Q(\lambda)h-\RE Q(\lambda_0)h=\widetilde Q(\lambda)h
\end{equation*}
for all $h\in\cG_0$ and $\lambda\in\dC\backslash\dR$, which is the first relation in ($\alpha$). 
Recall that for a generalized $Q$-function $Q(\bar\lambda)^*$
is an extension of $Q(\lambda)$. This implies $\RE Q(\lambda_0)\subset (\RE Q(\lambda_0))^*$,
\begin{equation*}
Q(\lambda)^*-\RE Q(\lambda_0)
\subset\bigl(Q(\lambda)-\RE Q(\lambda_0)\bigr)^*=\widetilde Q(\lambda)^*
\end{equation*}
and therefore also $Q(\lambda)^*h-\RE Q(\lambda_0)h=\widetilde Q(\lambda)^*h$ is true for all 
$h\in\cG_0$ and $\lambda\in\dC\backslash\dR$. Hence we have shown ($\alpha$).

Clearly $\widetilde Q$ in \eqref{qtilde} is
a holomorphic $\cL(\cG)$-valued function on $\dC\backslash\dR$. Denote by $\overline{\Gamma(\lambda)}$
the closure of $\Gamma(\lambda)=(I+(\lambda-\lambda_0)(A-\lambda)^{-1})\Gamma_{\lambda_0}$. Then
\begin{equation*}
\overline{\Gamma(\lambda)}=\bigl(I+(\lambda-\lambda_0)(A-\lambda)^{-1}\bigr)\overline \Gamma_{\lambda_0},
\qquad\lambda\in\dC\backslash\dR,
\end{equation*}
and it is not difficult to see that \eqref{q} extends to
\begin{equation*}
\widetilde Q(\lambda)-\widetilde Q(\mu)^*=(\lambda-\bar\mu)\Gamma(\mu)^*\overline{\Gamma(\lambda)}.
\end{equation*}
Hence
\begin{equation*}
\bigl(\IM\widetilde Q(\lambda)k,k\bigr)=
(\IM \lambda) \bigl(\Gamma(\lambda)^*\overline{\Gamma(\lambda)}k,k\bigr)
= (\IM \lambda) \Vert \overline{\Gamma(\lambda)}k\Vert^2
\end{equation*}
holds for all $k\in\cG$  and this implies that $\widetilde Q$ is a Nevanlinna function, cf. \eqref{imqpos}. 
Furthermore, for $h\in\cG_0$ we have
\begin{equation*}
\IM\widetilde Q(\lambda) h=(\IM\lambda) \Gamma(\lambda)^*\Gamma(\lambda) h
\end{equation*}
and from the property $\ker\Gamma(\lambda)=\{0\}$, cf. Lemma~\ref{gamlem}, we conclude that 
$\IM\widetilde Q(\lambda) h=0$ for $h\in\cG_0$ implies $h=0$, i.e., condition ($\beta$) holds. 
The same arguments as in \cite[Theorem 2.4, Corollaries 2.5 and 2.6]{LT77} together with
the assumption that $S$ is a densely defined closed simple symmetric operator show that $\widetilde Q$ satisfies 
the conditions in ($\gamma$).

\vskip 0.3cm\noindent
Let us now verify the converse direction. If $\widetilde Q$ is a $\cL(\cG)$-valued Nevanlinna function, 
$\lambda_0\in\dC\backslash\dR$
and the first condition in  ($\gamma$) holds,
then it is well known that there exists a Hilbert space $\cH$, a selfadjoint operator $A$ in $\cH$
and a mapping $\widetilde\Gamma\in\cL(\cG,\cH)$ such that the representation
\begin{equation}\label{qtilderep}
\widetilde Q(\lambda)=\RE\widetilde Q(\lambda_0)+\widetilde\Gamma^*\bigl((\lambda-\RE\lambda_0)+
(\lambda-\lambda_0)(\lambda-\overline\lambda_0)(A-\lambda)^{-1}\bigr)\widetilde\Gamma
\end{equation}
is valid for all $\lambda\in\dC\backslash\dR$, see, e.g., 
\cite{HSW98,LT77}. Furthermore,
the space $\cH$ can be chosen minimal, i.e., 
\begin{equation}\label{min}
\cH=\cspan\bigl\{\bigl(I+(\lambda-\lambda_0)(A-\lambda)^{-1}\bigr)\widetilde\Gamma k:
k\in\cG,\,\lambda\in\dC\backslash\dR\bigr\}.
\end{equation}
We define the mapping $\Gamma_{\lambda_0}$ to be the restriction of $\widetilde\Gamma$ onto $\cG_0$. 
As $\widetilde\Gamma$ is bounded the closure $\overline\Gamma_{\lambda_0}$ of $\Gamma_{\lambda_0}$ coincides with $\widetilde\Gamma$.
We claim that $\Gamma_{\lambda_0}$ is injective. In fact, if $\Gamma_{\lambda_0}h=0$ for some $h\in\cG_0$ then 
$\widetilde\Gamma h=0$ and by \eqref{qtilderep}
we have $\widetilde Q(\lambda) h=\RE\widetilde Q(\lambda_0)h$. Therefore $\IM\widetilde Q(\lambda)h=0$
and by assumption ($\beta$) this implies $h=0$.

Define the operator $S$ by
\begin{equation*}
Sf=Af,\quad \dom S=\bigl\{f\in\dom A: ((A-\bar\lambda_0)f,\Gamma_{\lambda_0}h)=0\,\,\text{for all}\,\, h\in\cG_0\bigr\}.
\end{equation*}
Then $S$ is a closed symmetric operator and  the identities
$\ran(S-\bar\lambda_0)=(\ran \Gamma_{\lambda_0})^\bot$ and $\ker(S^*-\lambda_0)=\overline{\ran\Gamma_{\lambda_0}}$ hold. 
Let 
\begin{equation}\label{gamlam}
\Gamma(\lambda)=(I+(\lambda-\lambda_0)(A-\lambda)^{-1})\Gamma_{\lambda_0},\qquad \lambda\in\dC\backslash\dR.
\end{equation} 
It is not difficult to check that
$\ran(S-\bar\lambda)=(\ran\Gamma(\lambda))^\bot$ is true for all $\lambda\in\dC\backslash\dR$ and the conditions 
in ($\gamma$) together with \eqref{min} now yield in the same way as in 
\cite[Theorem 2.4, Corollaries 2.5 and 2.6]{LT77} that $S$ is densely defined and simple.

Note that $\dom A\cap\ran \Gamma_{\lambda_0}=\{0\}$ since $\lambda_0\in\rho(A)$ and 
$\ran\Gamma_{\lambda_0}\subset\cN_{\lambda_0}(S^*)$.
Let us define a linear operator $T$ in $\cH$ on $\dom T:=\dom A\,\dot+\,\ran\Gamma_{\lambda_0}$ by
\begin{equation*}
T(f+f_{\lambda_0}):=Af + \lambda_0f_{\lambda_0},\qquad f\in\dom A,\,\, f_{\lambda_0}\in\ran\Gamma_{\lambda_0}.
\end{equation*}
Obviously $T$ is an extension of $A$ and since $\cN_{\lambda_0}(T)=\ran\Gamma_{\lambda_0}$ and $\ran\Gamma_{\lambda_0}$
is dense in $\cN_{\lambda_0}(S^*)$ we obtain from $\dom S^*=\dom A\,\dot+\,\cN_{\lambda_0}(S^*)$, cf. \eqref{decoall},
that $T\subset S^*$ and $\overline T=S^*$ holds. 

According to condition ($\alpha$) the Nevanlinna function $\widetilde Q$ and the function $Q$ are related by
\begin{equation*}
Q(\lambda)h=\widetilde Q(\lambda)h+\RE Q(\lambda_0)h\quad\text{and}\quad
Q(\lambda)^*h=\widetilde Q(\lambda)^*h+\RE Q(\lambda_0)h
\end{equation*}
for all $h\in\cG_0$ and $\lambda\in\dC\backslash\dR$. It remains to show that
$Q$ satisfies \eqref{q}. Observe first that for $\lambda,\mu\in\dC\backslash\dR$ we have
\begin{equation}\label{qcheck}
Q(\lambda)h-Q(\mu)^*h=\widetilde Q(\lambda)h-\widetilde Q(\mu)^*h.
\end{equation}
Denote the closures of the operators $\Gamma(\lambda)$, $\lambda\in\dC\backslash\dR$, 
in \eqref{gamlam} by $\widetilde\Gamma(\lambda)$.
Then 
\begin{equation*}
\widetilde\Gamma(\lambda)=\overline{\Gamma(\lambda)}=\bigl(I+(\lambda-\lambda_0)(A-\lambda)^{-1}\bigr)
\overline \Gamma_{\lambda_0}
=\bigl(I+(\lambda-\lambda_0)(A-\lambda)^{-1}\bigr)\widetilde\Gamma
\end{equation*}
and it follows from \eqref{qtilderep} with a straightforward calculation that
\begin{equation}\label{qtildegam}
 \widetilde Q(\lambda)-\widetilde Q(\mu)^*=(\lambda-\bar\mu)\widetilde\Gamma(\mu)^*\widetilde\Gamma(\lambda),
\qquad\lambda,\mu\in\dC\backslash\dR,
\end{equation}
holds. As $\widetilde\Gamma(\mu)^*=\overline{\Gamma(\mu)}^{\,*}=\Gamma(\mu)^*$ we conclude 
\begin{equation*}
Q(\lambda)h-Q(\mu)^*h=(\lambda-\bar\mu)\Gamma(\mu)^*\Gamma(\lambda)h,\qquad h\in\cG_0,
\end{equation*}
from \eqref{qcheck}. Therefore $Q$ is a generalized $Q$-function of the triple $\{S,A,T\}$.
\end{proof}

\begin{remark}
The definition of a generalized $Q$-function can be extended to the case that $A$ is a selfadjoint relation, $S$ is a
non-densely defined symmetric operator or relation and $T$ is a linear relation which is dense in the relation $S^*$.
We refer to \cite{LT77} for ordinary $Q$-functions in this more general situation.
In this case the condition {\rm ($\gamma$)} in Theorem~\ref{qthmgen1} can be dropped. 
\end{remark}

For ordinary $Q$-functions
Theorem~\ref{qthmgen1} reads as follows, cf. \cite[Theorem 2.2 and Theorem 2.4]{LT77}.

\begin{theorem}\label{qthmord}
A $\cL(\cG)$-valued Nevanlinna function $\widetilde Q$ is an ordinary $Q$-function of some pair $\{S,A\}$, where $S$ is a densely defined
closed simple symmetric operator in some Hilbert space $\cH$ and $A$ is a selfadjoint extension of $S$ in $\cH$, if and only if 
condition ($\gamma$) in Theorem~\ref{qthmgen1} and $0\in\rho(\IM \widetilde Q(\lambda))$ holds for some, and hence for all, $\lambda\in\dC\backslash\dR$.
\end{theorem}

\begin{corollary}\label{derivcor}
Let $Q$ be a generalized $Q$-function of $\{S,A,T\}$ and let $\widetilde Q$ be the 
$\cL(\cG)$-valued Nevanlinna function in Theorem~\ref{qthmgen1}.
Then for all $\lambda\in\dC\backslash\dR$ and  $h\in\cG_0$ we have
\begin{equation*}
 \frac{d}{d\lambda}\,Q(\lambda) h=\frac{d}{d\lambda}\,\widetilde Q(\lambda)h=\Gamma(\bar\lambda)^*\Gamma(\lambda)h.
\end{equation*}
\end{corollary}

\begin{proof}
It follows from \eqref{qtildegam} that
\begin{equation*}
\frac{d}{d\lambda}\,\widetilde Q(\lambda)=
\lim_{\bar\mu\rightarrow\lambda}\,\frac{\widetilde Q(\lambda)-\widetilde Q(\mu)^*}{\lambda-\bar\mu}
=\widetilde\Gamma(\bar\lambda)^*\widetilde\Gamma(\lambda)
\end{equation*}
holds. Hence condition ($\alpha$) in Theorem~\ref{qthmgen1} 
and $\widetilde\Gamma(\lambda)=\overline{\Gamma(\lambda)}$ imply
\begin{equation*}
\frac{d}{d\lambda}\,Q(\lambda)h
=\lim_{\bar\mu\rightarrow\lambda}\,\frac{Q(\lambda)h- Q(\mu)^*h}{\lambda-\bar\mu}
=\lim_{\bar\mu\rightarrow\lambda}\,\frac{\widetilde Q(\lambda)h-\widetilde Q(\mu)^*h}{\lambda-\bar\mu}
=\Gamma(\bar\lambda)^*\Gamma(\lambda)h
\end{equation*}
for $h\in\cG_0$.
\end{proof}

\section{Elliptic operators and the Dirichlet-to-Neumann map}\label{ellops}

Let $\Omega\subset\dR^n$ be a bounded or unbounded domain with compact $C^\infty$-boundary $\partial\Omega$. 
Let $\cL$ be the "formally selfadjoint" uniformly elliptic second order differential expression
\begin{equation}\label{cl}
(\cL f)(x):=-\sum_{j,k=1}^n \left( \frac{\partial}{\partial x_j} a_{jk} \frac{\partial f}{\partial x_k}\right)(x)+ a(x)f(x),
\end{equation}
$x\in\Omega$,
with bounded infinitely differentiable coefficients $a_{jk}\in C^\infty(\overline\Omega)$ satisfying $a_{jk}(x)=\overline{a_{kj}(x)}$ 
for all $x\in\overline\Omega$ and $j,k=1,\dots,n$, the function 
$a\in L^\infty(\Omega)$ is real valued and 
\begin{equation}\label{elliptic}
\sum_{j,k=1}^n a_{jk}(x)\xi_j\xi_k\geq C\sum_{k=1}^n\xi_k^2
\end{equation}
holds for some $C>0$, all $\xi=(\xi_1,\dots,\xi_n)^\top\in\dR^n$ and $x\in\overline\Omega$. We note that the assumptions on the 
domain $\Omega$ and the coefficients of $\cL$ can be relaxed but it is not our aim to treat the most general setting 
here. We refer the reader to e.g. \cite{G85,LM72,M,W} for possible generalizations.

In the following we consider the selfadjoint realizations of $\cL$ in $L^2(\Omega)$ subject to Dirichlet
and Neumann boundary conditions. For a function $f$ in the Sobolev space 
$H^2(\Omega)$ we denote the trace by 
$f\vert_{\partial\Omega}$ and the trace of the conormal derivative is defined by
\begin{equation*}
\frac{\partial f}{\partial\nu}\Bigl|_{\partial\Omega}:=\sum_{j,k=1}^n a_{jk} n_j \frac{\partial f}{\partial x_k}
\Bigl|_{\partial\Omega};
\end{equation*}
here $n(x)=(n_1(x),\dots, n_n(x))^\top$ is the unit vector at the point $x\in\partial\Omega$ pointing out of $\Omega$.
Recall that the mapping $C^\infty(\overline\Omega)\ni f\mapsto\bigl\{f|_{\partial\Omega},
\tfrac{\partial f}{\partial\nu}\bigl|_{\partial\Omega}\bigr\}$
extends by continuity to a continuous surjective mapping
\begin{equation}\label{tracemap}
 H^2(\Omega)\ni f\mapsto \left\{f|_{\partial\Omega},\frac{\partial f}{\partial\nu}\Bigl|_{\partial\Omega}\right\}
\in H^{3/2}(\partial\Omega)\times H^{1/2}(\partial\Omega).
\end{equation}
The kernel of this map is
\begin{equation*}
H^2_0(\Omega)=\left\{f\in H^2(\Omega): f\vert_{\partial\Omega}=\frac{\partial f}{\partial\nu}
\Bigl|_{\partial\Omega}=0\right\}
\end{equation*} 
which coincides with the closure of 
$C_0^\infty(\Omega)$ in $H^2(\Omega)$. We refer the reader to the monographs \cite{LM72,M,W} for 
more details. In the following the scalar products in $L^2(\Omega)$
and $L^2(\partial\Omega)$ are denoted by $(\cdot,\cdot)_\Omega$ and $(\cdot,\cdot)_{\partial\Omega}$, respectively.
Then Green`s identity
\begin{equation}\label{greenid}
(\cL f,g)_\Omega-(f,\cL g)_\Omega=
\left(f\vert_{\partial\Omega},\frac{\partial g}{\partial\nu}\Bigl|_{\partial\Omega}\right)_{\partial\Omega}
-\left(\frac{\partial f}{\partial\nu}\Bigl|_{\partial\Omega},g\vert_{\partial\Omega}
\right)_{\partial\Omega}
\end{equation}
holds for all functions $f,g\in H^2(\Omega)$. We note that \eqref{greenid} is even true for $f\in H^2(\Omega)$ and 
$g$ belonging to the domain of the maximal operator associated to $\cL$ in $L^2(\Omega)$ if the 
$(\cdot,\cdot)_{\partial\Omega}$ scalar product in $L^2(\partial\Omega)$ is extended by continuity to 
$H^{3/2}(\partial\Omega)\times H^{-3/2}(\partial\Omega)$ and
$H^{1/2}(\partial\Omega)\times H^{-1/2}(\partial\Omega)$,
respectively, see \cite{LM72,W}. However, we shall make use of \eqref{greenid} only for the case $f,g\in H^2(\Omega)$.

It is well known that the realizations $A_D$ and $A_N$ of $\cL$ subject to Dirichlet and Neumann boundary conditions 
defined by
\begin{equation}\label{adan}
\begin{split}
A_D f&=\cL f,\quad \dom A_D=\bigl\{f\in H^2(\Omega): f\vert_{\partial\Omega}=0\bigr\},\\
A_N f&=\cL f,\quad \dom A_N=\Bigl\{f\in H^2(\Omega): \frac{\partial f}{\partial\nu}\Bigl|_{\partial\Omega}=0\Bigr\},
\end{split}
\end{equation}
are selfadjoint operators in $L^2(\Omega)$. The following statement is known and can be found in, e.g., \cite{LM72}. 
It can be proved with similar
methods as Theorem~\ref{opscoup} in the next section.

\begin{proposition}\label{opprop}
Let $\cL$ be the elliptic differential expression in \eqref{cl}. Then the operator
\begin{equation}\label{minop}
Sf=\cL f,\qquad \dom S=H^2_0(\Omega),
\end{equation}
is a densely defined closed symmetric operator in $L^2(\Omega)$
with infinite deficiency indices $n_\pm(S)$ and the adjoint $S^*$ of $S$ coincides with the maximal operator 
associated to $\cL$,
\begin{equation*}
S^*f=\cL f,\qquad \dom S^*=\bigl\{f\in L^2(\Omega):\cL f\in L^2(\Omega)\bigr\}.
\end{equation*}
The operator 
\begin{equation*}
Tf=\cL f,\qquad \dom T=H^2(\Omega),
\end{equation*}
is not closed as an operator in $L^2(\Omega)$ and $T$ 
satisfies $\overline T=S^*$ and $T^*=S$. Furthermore, the selfadjoint operators 
$A_D$ and $A_N$ in \eqref{adan} are extensions of $S$ and restrictions of $T$.
\end{proposition}

In order to define a mapping $\Gamma_{\lambda_0}$ for the definition of a generalized $Q$-function
associated to the triple $\{S,A_D,T\}$ we make use of the decomposition \eqref{decoall} in the present situation.
More precisely, for all points $\lambda$ in the resolvent set $\rho(A_D)$ of the selfadjoint Dirichlet operator 
$A_D$ we have the direct sum decomposition of $\dom T=H^2(\Omega)$:
\begin{equation}\label{h2deco}
H^2(\Omega)=\dom A_D\,\dot +\,\cN_\lambda(T)=\bigl\{f\in H^2(\Omega):f\vert_{\partial\Omega}=0\bigr\}\,\dot +\,
\cN_\lambda(T),
\end{equation}
where 
\begin{equation*}
\cN_\lambda(T)=\ker(T-\lambda)=\bigl\{f_\lambda\in H^2(\Omega): \cL f_\lambda=\lambda f_\lambda \bigr\}.
\end{equation*}
Let now $\varphi$ be a function in $H^{3/2}(\partial\Omega)$ and let ${\lambda_0}\in\rho(A_D)$. 
Then it follows from \eqref{tracemap} and \eqref{h2deco} that there exists a unique function 
$f_{\lambda_0}\in H^2(\Omega)$ which solves the equation 
$\cL f_{\lambda_0}=\lambda_0 f_{\lambda_0}$, i.e., $f_{\lambda_0}\in\cN_{\lambda_0}(T)$, 
and satisfies $f_{\lambda_0}\vert_{\partial\Omega}=\varphi$. We shall denote the mapping that assigns 
$f_{\lambda_0}$ to $\varphi$ by
$\Gamma_{\lambda_0}$,
\begin{equation}\label{Gammalambda0}
H^{3/2}(\partial\Omega)\ni \varphi\mapsto \Gamma_{\lambda_0}\varphi :=f_{\lambda_0}\in\cN_{\lambda_0}(T),
\end{equation}
and we regard $\Gamma_{\lambda_0}$ as an operator from $L^2(\partial\Omega)$ into $L^2(\Omega)$
with $\dom \Gamma_{\lambda_0}=H^{3/2}(\partial\Omega)$ and $\ran\Gamma_{\lambda_0}=\cN_{\lambda_0}(T)$.

\begin{proposition}\label{Gammalambda0prop}
Let $\lambda_0\in\rho(A_D)$, let $\Gamma_{\lambda_0}$ be as in \eqref{Gammalambda0} and let $\lambda\in\rho(A_D)$. 
Then the following holds:
\begin{enumerate}
\item [{\rm (i)}] $\Gamma_{\lambda_0}$ is a bounded operator from $L^2(\partial\Omega)$ in $L^2(\Omega)$ with dense
domain $H^{3/2}(\partial\Omega)$;
\item [{\rm (ii)}] The operator $\Gamma(\lambda)=(I+(\lambda-\lambda_0)(A_D-\lambda)^{-1})\Gamma_{\lambda_0}$ is given
by 
\begin{equation*}
\Gamma(\lambda) \varphi=f_\lambda,\quad\text{where}\quad f_\lambda\in\cN_\lambda(T)\,\,\,\,\text{and}\,\,\,\,
f_\lambda\vert_{\partial\Omega}=\varphi;
\end{equation*}
\item [{\rm (iii)}] The mapping $\Gamma(\bar\lambda)^*:L^2(\Omega)\rightarrow L^2(\partial\Omega)$ satisfies 
\begin{equation*}
\Gamma(\bar\lambda)^*(A_D-\lambda)f=-\frac{\partial f}{\partial\nu}\Bigl|_{\partial\Omega},\qquad f\in\dom A_D.
\end{equation*}
\end{enumerate}
\end{proposition}

\begin{proof}
Statement (i) will be a consequence of (iii). We prove assertion (ii). Recall that by Lemma~\ref{gamlem}
the range of the operator $\Gamma(\lambda)$, $\lambda\in\rho(A_D)$, is $\cN_\lambda(T)$. Let
$\varphi\in \dom\Gamma(\lambda)=H^{3/2}(\partial\Omega)$ and choose elements $f_\lambda\in\cN_\lambda(T)$ and 
$f_{\lambda_0}\in\cN_{\lambda_0}(T)$ such that 
\begin{equation*}
f_\lambda\vert_{\partial\Omega}=\varphi=f_{\lambda_0}\vert_{\partial\Omega}
\end{equation*}
holds. According to \eqref{h2deco} the functions $f_\lambda$ and $f_{\lambda_0}$ are unique. Then 
$\Gamma_{\lambda_0}\varphi=f_{\lambda_0}$ and hence we obtain
\begin{equation*}
\Gamma(\lambda)\varphi=\Gamma_{\lambda_0}\varphi+(\lambda-\lambda_0)(A_D-\lambda)^{-1}\Gamma_{\lambda_0} \varphi
=f_{\lambda_0}+(\lambda-\lambda_0)(A_D-\lambda)^{-1}\Gamma_{\lambda_0} \varphi.
\end{equation*}
Since $(\lambda-\lambda_0)(A_D-\lambda)^{-1}\Gamma_{\lambda_0} \varphi$ belongs to $\dom A_D$ it 
is clear that the trace of this element vanishes. Therefore, the traces of the functions 
$\Gamma(\lambda)\varphi \in\cN_\lambda(T)$ and $f_{\lambda_0}$ coincide,
\begin{equation*}
(\Gamma(\lambda) \varphi)\vert_{\partial\Omega}=f_{\lambda_0}\vert_{\partial\Omega}=\varphi=f_\lambda\vert_{\partial\Omega}.
\end{equation*}
Thus we have that the traces of $\Gamma(\lambda)\varphi \in\cN_\lambda(T)$ and $f_\lambda\in\cN_\lambda(T)$ coincide
and from \eqref{h2deco} we conclude $\Gamma(\lambda)\varphi =f_\lambda$.
\vskip 0.3cm\noindent
(iii) Let $\varphi\in H^{3/2}(\partial\Omega)$ and choose the unique function $g_{\bar\lambda}\in\cN_{\bar\lambda}(T)$
with the property $g_{\bar\lambda}\vert_{\partial\Omega}=\varphi$. Hence we have $\Gamma(\bar\lambda)\varphi=g_{\bar\lambda}$
and for $f\in\dom A_D$ it follows
\begin{equation*}
\bigl(\Gamma(\bar\lambda)\varphi,(A_D-\lambda)f\bigr)_\Omega=
(g_{\bar\lambda},A_D f)_\Omega-(\bar\lambda g_{\bar\lambda},f)_\Omega=
(g_{\bar\lambda},A_D f)_\Omega-(T g_{\bar\lambda},f)_\Omega.
\end{equation*}
Making use of Green's identity \eqref{greenid} we find
\begin{equation*}
(g_{\bar\lambda},A_D f)_\Omega-(T g_{\bar\lambda},f)_\Omega=
\left(\frac{\partial g_{\bar\lambda}}{\partial\nu}\Bigl|_{\partial\Omega},f\vert_{\partial\Omega}\right)_{\partial\Omega}
-\left(g_{\bar\lambda}\vert_{\partial\Omega},
\frac{\partial f}{\partial\nu}\Bigl|_{\partial\Omega}\right)_{\partial\Omega}
\end{equation*}
and since the trace of $f\in\dom A_D$ vanishes the first summand on the right hand side is zero. Therefore
\begin{equation*}
\bigl(\Gamma(\bar\lambda)\varphi,(A_D-\lambda)f\bigr)_\Omega=-\left(g_{\bar\lambda}\vert_{\partial\Omega},
\frac{\partial f}{\partial\nu}\Bigl|_{\partial\Omega}\right)_{\partial\Omega}=\left(\varphi,-
\frac{\partial f}{\partial\nu}\Bigl|_{\partial\Omega}\right)_{\partial\Omega}
\end{equation*}
holds for all $\varphi\in\dom\Gamma(\bar\lambda)=H^{3/2}(\partial\Omega)$. This gives 
$(A_D-\lambda)f\in\dom\Gamma(\bar\lambda)^*$ and 
\begin{equation*}
\Gamma(\bar\lambda)^*(A_D-\lambda)f=-\frac{\partial f}{\partial\nu}\Bigl|_{\partial\Omega}.
\end{equation*}
Moreover, as $\lambda\in\rho(A_D)$ and $f\in\dom A_D$ was arbitrary we see that $\Gamma(\bar\lambda)^*$
is defined on the whole space $L^2(\Omega)$. This together with the fact that $\Gamma(\bar\lambda)^*$ is
closed implies 
\begin{equation*}
\Gamma(\bar\lambda)^*\in\cL\bigl(L^2(\Omega),L^2(\partial\Omega)\bigr)
\end{equation*}
for $\lambda\in\rho(A_D)$ and, in particular, 
$\Gamma(\bar\lambda)\subset\overline{\Gamma(\bar\lambda)}=\Gamma(\bar\lambda)^{**}$ is bounded. Inserting 
$\lambda_0=\bar\lambda$ this yields assertion (i).
\end{proof}

In the study of elliptic differential operators the so-called 
Dirichlet-to-Neumann map plays an important role, we mention only
\cite{AP04,BMNW08,GLMZ05,GMZ07,GMZ07-2,GM08,GM08-2,G68,M04,MPP07,MPR07,P07,P08,PR09,Post07,V52}.
Roughly speaking this operator maps the Dirichlet boundary value $f_\lambda\vert_{\partial\Omega}$ 
of an $H^2(\Omega)$-solution of the equation $\cL u=\lambda u$ onto the Neumann boundary
value $\tfrac{\partial f_\lambda}{\partial\nu}|_{\partial\Omega}$ of this solution.
In the following definition also a minus sign arises, which is needed to obtain a generalized $Q$-function
in Theorem~\ref{bigthm1}.
Otherwise $-Q$ would turn out to be a generalized $Q$-function.

\begin{definition}\label{dirneu}
Let $\lambda\in\rho(A_D)$ and assign to $\varphi\in H^{3/2}(\partial\Omega)$ the unique function 
$f_\lambda\in\cN_\lambda(T)$ such that $f_\lambda\vert_{\partial\Omega}=\varphi$, see \eqref{tracemap} and \eqref{h2deco}.
The operator $Q(\lambda)$ in $L^2(\partial\Omega)$ defined by
\begin{equation}\label{dnmap}
Q(\lambda) \varphi=Q(\lambda)(f_\lambda|_{\partial\Omega}):=-\frac{\partial f_\lambda}{\partial\nu}\Bigl|_{\partial\Omega},
\qquad \varphi\in \dom Q(\lambda)=H^{3/2}(\partial\Omega),
\end{equation}
is called the {\em Dirichlet-to-Neumann map} associated to $\cL$.
\end{definition}

Note that by \eqref{tracemap} the range of the Dirichlet-to-Neumann map $Q(\lambda)$, $\lambda\in\rho(A_D)$, 
lies in $H^{1/2}(\partial\Omega)$. We remark that the Dirichlet-to-Neumann map can be extended, e.g., to 
an operator from $H^1(\partial\Omega)$ in $L^2(\partial\Omega)$ if instead of $H^2(\Omega)$
the operator $T$ is defined on a suitable subspace of $H^{3/2}(\Omega)$, cf. \cite{AP04,BF62,B65,BL07,G71,LM72}.
However, for our purposes this is not necessary since $A_D$ and $A_N$ are defined on subspaces of $H^2(\Omega)$.

In the next theorem we show that the Dirichlet-to-Neumann map is a generalized $Q$-function
and we illustrate the usefulness of this object in the representation of the difference of the resolvents of the Dirichlet
and Neumann operators $A_D$ and $A_N$ in \eqref{adan}. Similar Krein type resolvent formulas can also
be found in \cite{BL07,BGW09,GM08,GM08-2,P08,PR09,Post07}.
The fact that the difference of the resolvents belongs to some
von Neumann-Schatten class depending on the dimension of the space is well-known and goes back to M.S.~Birman, 
cf. \cite{B62}.

\begin{theorem}\label{bigthm1}
Let $\cL$ be the elliptic differential expression in \eqref{cl} and let $A_D$ and $A_N$ be the selfadjoint
realizations of $\cL$ in \eqref{adan}. Denote by $S$ the minimal operator associated to $\cL$ 
and let $T=\cL\upharpoonright H^2(\Omega)$ be as in Proposition~\ref{opprop}. Define 
$\Gamma(\lambda)$ as in 
Proposition~\ref{Gammalambda0prop} and let $Q(\lambda)$, $\lambda\in\rho(A_D)$, be the Dirichlet-to-Neumann map.
Then the following holds:
\begin{enumerate}
\item [{\rm (i)}] $Q$ is a generalized $Q$-function of the triple $\{S,A_D,T\}$; 
\item [{\rm (ii)}] The operator $Q(\lambda)$ is injective 
for all $\lambda\in\rho(A_D)\cap\rho(A_N)$ and the resolvent formula
\begin{equation}\label{resform}
(A_D-\lambda)^{-1}-(A_N-\lambda)^{-1}=\Gamma(\lambda) Q(\lambda)^{-1} \Gamma(\bar\lambda)^*
\end{equation}
holds;
\item [{\rm (iii)}] For $p\in\dN$ and $2p+1>n$ the difference of the resolvents in \eqref{resform}
belongs to the von Neumann-Schatten class $\sS_p(L^2(\Omega))$.
\end{enumerate}
\end{theorem}

\begin{proof}
In order to proof assertion (i) we have to check the relation
\begin{equation}\label{qrel1}
Q(\lambda)-Q(\mu)^*=(\lambda-\bar\mu)\Gamma(\mu)^*\Gamma(\lambda),\qquad\lambda,\mu\in\rho(A_D),
\end{equation}
on $\dom Q(\lambda)\cap\dom Q(\mu)^*$. For this it will be first
shown that $\dom Q(\lambda)=H^{3/2}(\partial\Omega)$ is a subset of $\dom Q(\mu)^*$ and that $Q(\mu)^*$ is an extension of $Q(\bar\mu)$. 
Let $\psi\in H^{3/2}(\partial\Omega)$ and choose
the unique function $f_{\bar\mu}\in\cN_{\bar\mu}(T)$ such that $f_{\bar\mu}|_{\partial\Omega}=\psi$. 
For an arbitrary $\varphi\in\dom Q(\mu)=H^{3/2}(\partial\Omega)$ let $f_\mu\in\cN_\mu(T)$ be the unique function that satisfies 
$f_\mu|_{\partial\Omega}=\varphi$. By the definition of the Dirichlet-to-Neumann map we have
\begin{equation*}
Q(\mu)\varphi=-\frac{\partial f_\mu}{\partial\nu}\Bigl|_{\partial\Omega}\quad\text{and}\quad 
Q(\bar\mu)\psi=-\frac{\partial f_{\bar\mu}}{\partial\nu}\Bigl|_{\partial\Omega}
\end{equation*}
and hence Green's identity \eqref{greenid} shows
\begin{equation*}
\begin{split}
(Q(\mu)\varphi,\psi)_{\partial\Omega}&
=\left(-\frac{\partial f_\mu}{\partial\nu}\Bigl|_{\partial\Omega},f_{\bar\mu}|_{\partial\Omega} \right)_{\partial\Omega}\\
&=
\left(f_\mu|_{\partial\Omega},\frac{\partial f_{\bar\mu}}{\partial\nu}\Bigl|_{\partial\Omega} \right)_{\partial\Omega}-
\left(\frac{\partial f_\mu}{\partial\nu}\Bigl|_{\partial\Omega},f_{\bar\mu}|_{\partial\Omega} \right)_{\partial\Omega}
+
\left(\varphi,-\frac{\partial f_{\bar\mu}}{\partial\nu}\Bigl|_{\partial\Omega} \right)_{\partial\Omega}\\
&=(Tf_\mu,f_{\bar\mu})_\Omega-(f_\mu,T f_{\bar\mu})_\Omega+\left(\varphi,-\frac{\partial f_{\bar\mu}}{\partial\nu}\Bigl|_{\partial\Omega} \right)_{\partial\Omega}.
\end{split}
\end{equation*}
Since $f_\mu\in\cN_\mu(T)$ and $f_{\bar\mu}\in\cN_{\bar\mu}(T)$ it is clear that $(Tf_\mu,f_{\bar\mu})_\Omega=(f_\mu,T f_{\bar\mu})_\Omega$ 
holds and therefore we obtain
\begin{equation*}
 (Q(\mu)\varphi,\psi)_{\partial\Omega}=\left(\varphi,-\frac{\partial f_{\bar\mu}}{\partial\nu}\Bigl|_{\partial\Omega} \right)_{\partial\Omega}
\end{equation*}
for all $\varphi\in\dom Q(\mu)$. Thus $\psi\in\dom Q(\mu)^*$ and 
\begin{equation*}
Q(\mu)^*\psi=-\frac{\partial f_{\bar\mu}}{\partial\nu}\Bigl|_{\partial\Omega}=Q(\bar\mu)\psi.
\end{equation*}

Next we prove the relation \eqref{qrel1}. Let $\varphi,\psi\in H^{3/2}(\partial\Omega)$ and
choose the functions $f_\lambda\in\cN_\lambda(T)$ and $g_\mu\in\cN_\mu(T)$ such that
$f_\lambda\vert_{\partial\Omega}=\varphi$ and $g_\mu\vert_{\partial\Omega}=\psi$.
Hence we have 
\begin{equation*}
Q(\lambda)\varphi=-\frac{\partial f_\lambda}{\partial\nu}\Bigl|_{\partial\Omega},\quad 
Q(\mu)\psi=-\frac{\partial g_\mu}{\partial\nu}\Bigl|_{\partial\Omega},\quad
\Gamma(\lambda)\varphi=f_\lambda\quad\text{and}\quad\Gamma(\mu)\psi=g_\mu.
\end{equation*}
Note that $\varphi\in H^{3/2}(\Omega)$ belongs to $\dom Q(\mu)^*$ by the above considerations.
With the help of Green's identity \eqref{greenid} we find
\begin{equation*}
\begin{split}
\bigl((Q(\lambda)&-Q(\mu)^*)\varphi,\psi\bigr)_{\partial\Omega}=-\left(\frac{\partial f_\lambda}{\partial\nu}\Bigl|_{\partial\Omega},
g_\mu\vert_{\partial\Omega}\right)_{\partial\Omega}+\left(f_\lambda\vert_{\partial\Omega},
\frac{\partial g_\mu}{\partial\nu}\Bigl|_{\partial\Omega}\right)_{\partial\Omega}\\
&=(T f_\lambda,g_\mu)_\Omega - (f_\lambda, Tg_\mu)_\Omega=(\lambda-\bar\mu)(f_\lambda,g_\mu)_\Omega\\
&=(\lambda-\bar\mu)(\Gamma(\lambda)\varphi,\Gamma(\mu)\psi)_\Omega
=\bigl((\lambda-\bar\mu)\Gamma(\mu)^*\Gamma(\lambda)\varphi,\psi\bigr)_{\partial\Omega}.
\end{split}
\end{equation*}
This holds for all $\psi$ in the dense subset $H^{3/2}(\partial\Omega)$ of $L^2(\partial\Omega)$ and therefore
\eqref{qrel1} is valid on $\dom Q(\lambda)=\dom\Gamma(\lambda)=H^{3/2}(\partial\Omega)$, i.e., the 
Dirichlet-to-Neumann map is a generalized $Q$-function of the triple $\{S,A_D,T\}$.

\vskip 0.3cm\noindent
(ii) Let $\lambda\in\rho(A_D)\cap\rho(A_N)$ and suppose that we have $Q(\lambda)\varphi=0$ for 
some $\varphi\in H^{3/2}(\partial\Omega)$. There exists a unique $f_\lambda\in\cN_\lambda(T)$ such that
$f_\lambda\vert_{\partial\Omega}=\varphi$ and for this $f_\lambda$ by assumption we have 
$\tfrac{\partial f_\lambda}{\partial\nu}|_{\partial\Omega}=0$. Hence $f_\lambda\in\dom A_N\cap\cN_\lambda(T)$
and from $\lambda\in\rho(A_N)$ we conclude $f_\lambda=0$, that is, $\varphi=f_\lambda|_{\partial\Omega}=0$.

Therefore $Q(\lambda)^{-1}$, $\lambda\in\rho(A_D)\cap\rho(A_N)$ exists and, roughly speaking, $Q(\lambda)^{-1}$
maps the negative Neumann boundary values of $H^2(\Omega)$-solutions of $\cL u=\lambda u$ onto their Dirichlet boundary values.
Let us proof the formula \eqref{resform} for the difference of the resolvents of $A_D$ and $A_N$. Observe first,
that the right hand side in \eqref{resform} is well defined. In fact, by Proposition~\ref{Gammalambda0prop}~(iii)
and \eqref{tracemap} the range of $\Gamma(\bar\lambda)^*$ lies in $H^{1/2}(\partial\Omega)$ and it follows from the surjectivity of the 
mapping in \eqref{tracemap} that $Q(\lambda)^{-1}$ is defined on the whole space $H^{1/2}(\partial\Omega)$ and maps 
$H^{1/2}(\partial\Omega)$ onto $H^{3/2}(\partial\Omega)$, the domain of $\Gamma(\lambda)$. 

Let now $f\in L^2(\Omega)$. 
We claim that the function
\begin{equation}\label{function}
g=(A_D-\lambda)^{-1}f-\Gamma(\lambda)Q(\lambda)^{-1}\Gamma(\bar\lambda)^* f
\end{equation}
belongs to $\dom A_N$. It is clear that $g$ is in $H^2(\Omega)$ since $(A_D-\lambda)^{-1}f\in\dom A_D$ and
the second term on the right hand side belongs to $\cN_\lambda(T)$, the range of $\Gamma(\lambda)$. In order
to verify $\tfrac{\partial g}{\partial \nu}|_{\partial\Omega}=0$ we choose $f_D\in\dom A_D$ such that
$f=(A_D-\lambda)f_D$, so that \eqref{function} becomes 
\begin{equation}\label{function2}
g=f_D-\Gamma(\lambda)Q(\lambda)^{-1}\Gamma(\bar\lambda)^* (A_D-\lambda)f_D
= f_D+\Gamma(\lambda)Q(\lambda)^{-1}\frac{\partial f_D}{\partial\nu}\Bigl|_{\partial\Omega},
\end{equation} 
where we have used Proposition~\ref{Gammalambda0prop}~(iii).
Let $f_\lambda:=\Gamma(\lambda)Q(\lambda)^{-1}\tfrac{\partial f_D}{\partial\nu}|_{\partial\Omega}$. 
Then $f_\lambda\in\cN_\lambda(T)$ and the trace of $f_\lambda$ is given by
\begin{equation*}
f_\lambda|_{\partial\Omega}=Q(\lambda)^{-1}\frac{\partial f_D}{\partial\nu}\Bigl|_{\partial\Omega}.
\end{equation*}
Hence $Q(\lambda)f_\lambda|_{\partial\Omega}=\tfrac{\partial f_D}{\partial\nu}|_{\partial\Omega}$, but on the other hand,
by the definition of the Dirichlet-to-Neumann map  
$Q(\lambda)f_\lambda|_{\partial\Omega}=-\tfrac{\partial f_\lambda}{\partial\nu}|_{\partial\Omega}$. Therefore,
the sum of the Neumann boundary value of the function $f_\lambda$ and the Neumann boundary value of $f_D$
is zero and we conclude from \eqref{function2}
\begin{equation*}
\frac{\partial g}{\partial\nu}\Bigl|_{\partial\Omega}=
\frac{\partial f_D}{\partial\nu}\Bigl|_{\partial\Omega}+\frac{\partial f_\lambda}{\partial\nu}\Bigl|_{\partial\Omega}=0.
\end{equation*}
We have shown that $g$ in \eqref{function} belongs to $\dom A_N$. As $T$ is an extension of $A_N$ and $A_D$,
and $\ran\Gamma(\lambda)=\ker(T-\lambda)$  we obtain
\begin{equation*}
(A_N-\lambda) g=(T-\lambda)(A_D-\lambda)^{-1}f-(T-\lambda)\Gamma(\lambda)Q(\lambda)^{-1}\Gamma(\bar\lambda)^* f
=f.
\end{equation*}
Together with \eqref{function} we find 
\begin{equation*}
(A_N-\lambda)^{-1}f=(A_D-\lambda)^{-1}f-\Gamma(\lambda)Q(\lambda)^{-1}\Gamma(\bar\lambda)^* f
\end{equation*}
for all $\lambda\in\rho(A_D)\cap\rho(A_N)$ and $f\in L^2(\Omega)$, and therefore the resolvent formula
\eqref{resform} is valid.

\vskip 0.3cm\noindent
Up to some small modifications assertion (iii) was proved in \cite{B62}.
\end{proof}

We mention that for $\lambda,\lambda_0\in\rho(A_D)$ the Dirichlet-to-Neumann map is connected with 
the resolvent of $A_D$ via
 \begin{equation*}
 Q(\lambda)=\RE Q(\lambda_0)+\Gamma_{\lambda_0}\bigl((\lambda-\RE\lambda_0)+
(\lambda-\lambda_0)(\lambda-\bar\lambda_0)(A_D-\lambda)^{-1}\bigr)\Gamma_{\lambda_0}.
 \end{equation*}
This follows from the fact that $Q$ is a generalized $Q$-function and Proposition~\ref{formq}.
The following two corollaries collect some properties of the Dirichlet-to-Neumann map
and its inverse. 

\begin{corollary}\label{prop1}
For $\lambda,\lambda_0\in\rho(A_D)$ the Dirichlet-to-Neumann map $Q(\lambda)$ has the following properties.
\begin{enumerate}
 \item [{\rm (i)}] $Q(\lambda)$ is a non-closed unbounded operator in $L^2(\partial\Omega)$ defined
                   on $H^{3/2}(\partial\Omega)$ with $\ran Q(\lambda)\subset H^{1/2}(\partial\Omega)$;
                   
 \item [{\rm (ii)}] $Q(\lambda)-\RE Q(\lambda_0)$ is a non-closed bounded operator in $L^2(\partial\Omega)$ defined
                   on $H^{3/2}(\partial\Omega)$;
 \item [{\rm (iii)}] the closure $\widetilde Q(\lambda)$ of the operator $Q(\lambda)-\RE Q(\lambda_0)$ in 
$L^2(\partial\Omega)$ satisfies 
$$\frac{d}{d\lambda}\,\widetilde Q(\lambda)=\Gamma(\bar\lambda)^*\overline{\Gamma(\lambda)}$$
and $\widetilde Q$ is a $\cL(L^2(\partial\Omega))$-valued
 Nevanlinna function. 
 \end{enumerate}
\end{corollary}

\begin{proof}
Besides the statement that $Q(\lambda)$ is a non-closed unbounded operator the assertions follow from the fact that 
$Q$ is a generalized $Q$-function and the results in Section~\ref{genq}. In Corollary~\ref{prop2} it will turn out that $\overline{Q(\lambda)^{-1}}$
is a compact operator and that $Q(\lambda)^{-1}$ is not closed. This implies that $\overline{Q(\lambda)}$ and
$Q(\lambda)$ are unbounded and that $Q(\lambda)$ is not closed. 
\end{proof}

\begin{corollary}\label{prop2}
For $\lambda\in\rho(A_D)\cap\rho(A_N)$ the inverse $Q(\lambda)^{-1}$ of the Dirichlet-to-Neumann map $Q(\lambda)$ has the following properties.
\begin{enumerate}
 \item [{\rm (i)}] $Q(\lambda)^{-1}$ is a non-closed bounded operator in $L^2(\partial\Omega)$ defined
                    on $H^{1/2}(\partial\Omega)$ with $\ran Q(\lambda)^{-1}=H^{3/2}(\partial\Omega)$;
 \item [{\rm (ii)}] the closure $\overline{Q(\lambda)^{-1}}$ is a compact operator in $L^2(\partial\Omega)$;
 \item [{\rm (iii)}] the function $\lambda\mapsto -\overline{Q(\lambda)^{-1}}$ is a $\cL(L^2(\partial\Omega))$-valued
 Nevanlinna function.
 \end{enumerate}
\end{corollary}

\begin{proof}
It is clear that (i) is an immediate consequence of (ii). Statement (iii) follows from Theorem~\ref{qthmgen1} and 
general properties of the Nevanlinna class.
Assertion (ii) is essentially a consequence of the classical results in \cite{LM72}, see also
\cite[Theorem~2.1]{G71}. Namely, for $\lambda\in\rho(A_D)\cap\rho(A_N)$ the operator
$Q(\lambda):H^{3/2}(\partial\Omega)\rightarrow H^{1/2}(\partial\Omega)$ is an isomorphism and 
can be extended to an isomorphism $\widehat Q(\lambda):H^{1}(\partial\Omega)\rightarrow L^2(\partial\Omega)$ 
which acts as in \eqref{dnmap}. Therefore $Q(\lambda)^{-1}\subset \widehat Q(\lambda)^{-1}$ is a densely defined
operator in $L^2(\partial\Omega)$ which is bounded as an operator in $H^1(\partial\Omega)$
and hence also bounded when considered as an operator in $L^2(\partial\Omega)$. Its closure $\overline{Q(\lambda)^{-1}}$
in $L^2(\partial\Omega)$ 
is a bounded everywhere defined operator in $L^2(\partial\Omega)$ with values in $H^1(\partial\Omega)$
and coincides with $\widehat Q(\lambda)^{-1}$. 
As $H^1(\partial\Omega)$ is compactly embedded in $L^2(\partial\Omega)$ it follows that $\overline{Q(\lambda)^{-1}}$
is a compact operator in $L^2(\partial\Omega)$.
\end{proof}

The next corollary is a simple consequence of Theorem~\ref{bigthm1} 
for the case that the difference of the resolvents is a trace class operator.

\begin{corollary}
Let the assumptions be as in Theorem~\ref{bigthm1}, let $\widetilde Q$ be the Nevanlinna function
from Corollary~\ref{prop1} and suppose, in addition, $n=2$. Then 
\begin{equation}\label{traceresform}
\tr\bigl((A_D-\lambda)^{-1}-(A_N-\lambda)^{-1}\bigr)=\tr\left(\overline{Q(\lambda)^{-1}}\,\,
\frac{d}{d\lambda}\,\widetilde Q(\lambda)\right)
\end{equation}
holds for all $\lambda\in\rho(A_D)\cap\rho(A_N)$. 
\end{corollary}

\begin{proof}
The resolvent formula \eqref{resform} can be written in the form
\begin{equation}\label{resform234}
(A_D-\lambda)^{-1}-(A_N-\lambda)^{-1}=\overline{\Gamma(\lambda)}\,\overline{Q(\lambda)^{-1}}\,\Gamma(\bar\lambda)^*,
\end{equation}
where the closures $\overline{\Gamma(\lambda)}$ and $\overline{Q(\lambda)^{-1}}$ are everywhere defined bounded operators,
cf. Corollary~\ref{prop2}~(ii). In the case $n=2$ it follows from Theorem~\ref{bigthm1}~(iii) that \eqref{resform234}
is a trace class operator and from Corollaries~\ref{derivcor}, \ref{prop1}~(iii) and well known properties 
of the trace of bounded operators (see \cite{GK69}) we conclude \eqref{traceresform}.
\end{proof}

\section{Coupling of elliptic differential operators}\label{cellops}

In this section we study the uniformly elliptic second order differential expression $\cL$
from \eqref{cl} on two different domains and a coupling of the associated Dirichlet operators.
More precisely, let $\Omega\subset\dR^n$ be a simply connected bounded domain with $C^\infty$-boundary $\cC:=\partial\Omega$
and let $\Omega^\prime=\dR^n\backslash\overline\Omega$ be the complement of the closure of $\Omega$ in $\dR^n$.
Clearly, $\Omega^\prime$ is an unbounded domain with the compact $C^\infty$-boundary 
$\partial\Omega^\prime=\cC$. Let again $\cL$ be given by
\begin{equation}\label{cl2}
\cL h=-\sum_{j,k=1}^n  \frac{\partial}{\partial x_j}\, a_{jk} \frac{\partial h}{\partial x_k} + ah
\end{equation}
with bounded coefficients $a_{jk}\in C^\infty(\dR^n)$ satisfying $a_{jk}(x)=\overline{a_{kj}(x)}$ 
for all $x\in\dR^n$ and $j,k=1,\dots,n$, the function $a\in L^\infty(\dR^n)$ is real valued and 
suppose that $\cL$ is uniformly elliptic, cf. \eqref{elliptic}. The restriction of $\cL$ on functions $f$ defined 
on $\Omega$ or functions $f^\prime$ defined on $\Omega^\prime$ will be denoted by $\cL_\Omega$ and $\cL_{\Omega^\prime}$, respectively. Then
it is clear that the differential expressions $\cL_\Omega$ and $\cL_{\Omega^\prime}$ are of the type
as in Section~\ref{ellops}. 

In the following we will usually denote functions defined on $\dR^n$ by $h$ or $k$, 
and we denote functions defined on $\Omega$ or $\Omega^\prime$
by $f,g$ or $f^\prime,g^\prime$, respectively. The scalar products of $L^2(\Omega)$ and $L^2(\Omega^\prime)$
are indexed with $\Omega$ and $\Omega^\prime$, respectively, whereas the scalar product of $L^2(\dR^n)$ is just
denoted by $(\cdot,\cdot)$. For the trace of a function $f\in H^2(\Omega)$ and $f^\prime\in H^2(\Omega^\prime)$
we write $f|_\cC$ and $f^\prime|_\cC$, and the trace of the conormal derivatives are
\begin{equation}\label{cono}
\frac{\partial f}{\partial\nu}\Bigl|_\cC=\sum_{j,k=1}^n a_{jk}n_j\,\frac{\partial f}{\partial x_k}\Bigl|_\cC\quad\text{and}\quad
\frac{\partial f^\prime}{\partial\nu^\prime}\Bigl|_\cC=\sum_{j,k=1}^n a_{jk}n^\prime_j\,
\frac{\partial f}{\partial x_k}\Bigl|_\cC;
\end{equation}
here $n(x)=(n_1(x),\dots,n_n(x))^\top$ and $n^\prime(x)=-n(x)$ are the unit vectors at the point $x\in\cC=\partial\Omega=\partial\Omega^\prime$ pointing out
of $\Omega$ and $\Omega^\prime$, respectively. Note also that the coefficients $a_{jk}$ in \eqref{cono}
are the restrictions of the coefficients in \eqref{cl2} onto $\Omega$ and $\Omega^\prime$, respectively. 
The Dirichlet operators 
\begin{equation*}
\begin{split}
A_\Omega f&=\cL_\Omega f,\qquad\,\,\,\, \,\,\dom A_\Omega=\bigl\{f\in H^2(\Omega):f|_\cC=0\bigr\},\\
A_{\Omega^\prime}f^\prime&=\cL_{\Omega^\prime}f^\prime,\qquad \dom A_{\Omega^\prime}=\bigl\{f^\prime\in H^2(\Omega^\prime):f^\prime|_\cC=0\bigr\},
\end{split}
\end{equation*}
are selfadjoint operators in $L^2(\Omega)$ and $L^2(\Omega^\prime)$, respectively. Hence the 
orthogonal sum 
\begin{equation}\label{aschro}
A=\begin{pmatrix} A_\Omega & 0 \\ 0 & A_{\Omega^\prime}\end{pmatrix},\qquad \dom A=\dom A_\Omega\oplus \dom A_{\Omega^\prime},
\end{equation}
is a selfadjoint operator in $L^2(\dR^n)=L^2(\Omega)\oplus L^2(\Omega^\prime)$. Observe that
\begin{equation}\label{aschroe}
\begin{split}
A(f\oplus f^\prime)&=\cL (f\oplus f^\prime)=\cL_\Omega f\oplus \cL_{\Omega^\prime}f^\prime,\\ 
\dom A&=\bigl\{f\oplus f^\prime \in H^2(\Omega)\oplus H^2(\Omega^\prime):f|_\cC=0=f^\prime|_\cC\bigr\},
\end{split}
\end{equation}
and that $A$ 
is not a usual second order elliptic differential operator on $\dR^n$ since for a function 
$f\oplus f^\prime\in \dom A$ the traces of the conormal derivatives $\tfrac{\partial f}{\partial\nu}|_\cC$
and $-\tfrac{\partial f^\prime}{\partial\nu^\prime}|_\cC$ at the boundary $\cC$ of the domains $\Omega$ and $\Omega^\prime$ 
in general do not coincide. 

Besides the operator $A$ we consider the usual selfadjoint operator associated to $\cL$ in $L^2(\dR^n)$ defined
by
\begin{equation}\label{atildeschroe}
\widetilde A h=\cL h,\qquad h\in\dom\widetilde A=H^2(\dR^n),
\end{equation}
and our aim is to prove a formula for the difference of the resolvents of $\widetilde A$ and $A$ 
with the help of a generalized $Q$-function in a similar form as in the
previous section.
 
The following theorem indicates how $S$ and $T$ in the triple $\{S,A,T\}$ for the definition of a generalized 
$Q$-function can be chosen. 

\begin{theorem}\label{opscoup}
The operator 
\begin{equation}
S h=\cL h,\quad \dom S=\bigl\{h=f\oplus f^\prime \in H^2(\dR^n):f|_\cC=0=f^\prime|_\cC\bigr\},
\end{equation}
is a densely defined closed 
symmetric operator in $L^2(\dR^n)$ with infinite deficiency indices $n_\pm(S)$. The operator
\begin{equation}
\begin{split}
T(f\oplus f^\prime)&=\cL(f\oplus f^\prime),\\
\dom T&=\bigl\{f\oplus f^\prime \in H^2(\Omega)\oplus H^2(\Omega^\prime):f|_\cC=f^\prime|_\cC\bigr\},
\end{split}
\end{equation}
is not closed as an operator in $L^2(\dR^n)$ and $T$ satisfies $\overline T=S^*$ and $T^*=S$. Furthermore,
the selfadjoint operators $A$ and $\widetilde A$ in \eqref{aschro}, \eqref{aschroe} and \eqref{atildeschroe} are 
extensions of $S$ and restrictions of $T$.
\end{theorem}

\begin{proof}
The operator $S$ is a restriction of the selfadjoint operator $A$ and hence $S$ is symmetric. The fact that
$\dom S$ is dense follows, e.g., from the fact that $H_0^2(\Omega)$ and $H^2_0(\Omega^\prime)$ are dense 
subspaces of $L^2(\Omega)$ and $L^2(\Omega^\prime)$, respectively, cf. Proposition~\ref{opprop}, and 
$$H_0^2(\Omega)\oplus H^2_0(\Omega^\prime)\subset\dom S.$$
Since for any function $h\in H^2(\dR^n)$ decomposed as $h=f\oplus f^\prime$, where $f\in H^2(\Omega)$, $f^\prime\in H^2(\Omega^\prime)$,
we have $f\vert_\cC=f^\prime\vert_\cC\in H^{3/2}(\cC)$ it follows that $\widetilde A$ is an extension of $S$ 
and a restriction of the operator $T$. Moreover, $S\subset A\subset T$ is obvious.

Let us verify that $S=T^*$ holds. In particular this implies that $S$ is closed and that $\overline T=S^*$ is true.
We start with the inclusion $S\subset T^*$. Let $h=f\oplus f^\prime\in\dom S$ and $k= g \oplus g^\prime\in\dom T$, 
where $f,g\in H^2(\Omega)$ and 
$f^\prime , g^\prime \in H^2(\Omega^\prime)$. First of all we have 
\begin{equation*}
(T k,h)-(k,Sh)=(\cL_\Omega g,f)_\Omega-(g,\cL_\Omega f)_\Omega+
(\cL_{\Omega^\prime} g^\prime ,f^\prime)_{\Omega^\prime}-(g^\prime,\cL_{\Omega^\prime} f^\prime)_{\Omega^\prime}
\end{equation*}
and Green's identity \eqref{greenid} shows that this is equal to
\begin{equation*}
\left(g\vert_\cC,\frac{\partial f}{\partial\nu}\Bigl|_\cC\right)_\cC-
\biggl(\frac{\partial g}{\partial\nu}\Bigl|_\cC,f\vert_\cC\biggr)_\cC+
\left(g^\prime \vert_\cC,\frac{\partial f^\prime}{\partial\nu^\prime}\Bigl|_\cC\right)_\cC-
\left(\frac{\partial g^\prime}{\partial\nu^\prime}\Bigl|_\cC,f^\prime\vert_\cC\right)_\cC.
\end{equation*}
Since $h=f\oplus f^\prime\in\dom S$ we have
\begin{equation*}
f\vert_\cC=f^\prime\vert_\cC=0\qquad\text{and}\qquad\frac{\partial f}{\partial\nu}\Bigl|_\cC=
-\frac{\partial f^\prime}{\partial\nu^\prime}\Bigl|_\cC,
\end{equation*}
and for $k=g\oplus g^\prime\in\dom T$ by definition $g\vert_\cC=g^\prime\vert_\cC$ holds.
Hence we conclude
\begin{equation*}
(T k,h)-(k,Sh)=0
\end{equation*}
and therefore every $h\in\dom S$ belongs to $\dom T^*$ and $T^*h=Sh$, i.e., $S\subset T^*$. Let us now prove the converse
inclusion $T^*\subset S$. For this it is sufficient to check that every function $h\in\dom T^*$ belongs to 
$\dom S$. From the fact that
$T$ is an extension of the selfadjoint operators $A$ and $\widetilde A$ we conclude 
\begin{equation*}
T^*\subset A^*=A\subset T\qquad\text{and}\qquad T^*\subset\widetilde A^*=\widetilde A\subset T,
\end{equation*}
so
that $T^*$ is a restriction of $A$ and $\widetilde A$. Hence every function $h$ in $\dom T^*$ belongs also to
$\dom A$ and $\dom \widetilde A$. Thus $h=f\oplus f^\prime\in H^2(\dR^n)$ and $f\in H^2(\Omega)$ and $f^\prime\in H^2(\Omega^\prime)$
satisfy $f\vert_\cC=f^\prime\vert_\cC=0$. Therefore $\dom T^*\subset \dom S$
and we have shown $T^*=S$.

Next it will be verified that $T$ is not closed. The arguments are similar as in \cite[Proof of Proposition 4.5]{BKSZ08} and could 
also be formulated in terms of unitary relations
between Krein spaces, cf. \cite{DHMS06}. Assume that $T$ is closed, i.e., $T=\overline T$, and consider the subspace
\begin{equation*}
\cM=\left\{\left[\begin{matrix}f\oplus f^\prime \\ T(f\oplus f^\prime) \\ f\vert_\cC\\ \tfrac{\partial f}{\partial \nu}|_\cC + 
\tfrac{\partial f^\prime}{\partial \nu^\prime}|_\cC\end{matrix}\right]:
 f\oplus f^\prime \in \dom T 
\right\}\subset L^2(\dR^n)\oplus L^2(\dR^n)\oplus L^2(\cC)
\oplus L^2(\cC).
\end{equation*}
Observe that by \eqref{tracemap} and the definition of $T$ the mapping
\begin{equation}\label{tracemapt}
\dom T\ni f \oplus f^\prime\,\,\mapsto\,\,\left\{f|_\cC,\frac{\partial f}{\partial\nu}\Bigl|_\cC+\frac{\partial f^\prime}{\partial\nu^\prime}\Bigl|_\cC\right\}
\in H^{3/2}(\cC)\,\times\, H^{1/2}(\cC)
\end{equation}
is onto.
Setting $\cN=L^2(\Omega)\oplus L^2(\Omega)\oplus \{0\}\oplus \{0\}$ it is clear that the sum of the 
subpaces $\cM$ and $\cN$ is
\begin{equation}\label{cmcn}
\cM+\cN=L^2(\dR^n)\oplus L^2(\dR^n)\oplus \bigl( H^{3/2}(\cC)\,\times\, H^{1/2}(\cC)\bigr).
\end{equation}
We will calculate the orthogonal complements of $\cM$ and $\cN$ in $L^2(\dR^n)\oplus L^2(\dR^n)\oplus L^2(\cC)
\oplus L^2(\cC)$ and show that $\cM^\bot +\cN^\bot$ is closed. First of all we have
\begin{equation}\label{nbot}
\cN^\bot=\{0\}\oplus \{0\}\oplus L^2(\cC)\oplus L^2(\cC)
\end{equation}
and in order to determine $\cM^\bot$ suppose that 
\begin{equation}\label{mort}
\left[\begin{matrix} l \oplus l^\prime \\  g\oplus g^\prime \\ \varphi \\ \psi\end{matrix}\right]\in\cM^\bot,\qquad g,l \in L^2(\Omega),
\,\, g^\prime,l^\prime\in L^2(\Omega^\prime),\,\,\varphi,\psi\in L^2(\cC),
\end{equation} 
is an element in $L^2(\dR^n)\oplus L^2(\dR^n)\oplus L^2(\cC)
\oplus L^2(\cC)$ which is orthogonal to $\cM$.
Then we have 
\begin{equation}\label{polk}
\bigl(T(f\oplus f^\prime),g\oplus g^\prime\bigr)+\bigl(f \oplus f^\prime, l \oplus l^\prime \bigr)=-\bigl(f\vert_\cC,\varphi\bigr)_\cC-
\left(\frac{\partial f}{\partial \nu}\Bigl|_\cC + 
\frac{\partial f^\prime}{\partial \nu^\prime}\Bigl|_\cC,\psi\right)_\cC
\end{equation}
for all $f\oplus f^\prime\in \dom T$. In particular, for $f\oplus f^\prime\in \dom S$ we have 
\begin{equation*}
\frac{\partial f}{\partial \nu}\Bigl|_\cC = - \frac{\partial f^\prime}{\partial \nu^\prime}\Bigl|_\cC\quad\text{and}\quad 
f\vert_\cC=f^\prime\vert_\cC=0,
\end{equation*}
so that \eqref{polk} becomes
\begin{equation*}
\bigl(T(f\oplus f^\prime),g\oplus g^\prime\bigr)=\bigl(S(f\oplus f^\prime),g\oplus g^\prime\bigr)=-\bigl(f \oplus f^\prime,  l \oplus l^\prime\bigr)
\end{equation*}
and hence $g\oplus g^\prime\in \dom S^*$ and $S^*(g\oplus g^\prime)=- l \oplus l^\prime$. But we have assumed that $T$ is closed and hence
from $S=T^*$ we conclude $S^*=T^{**}=\overline T=T$,
so that 
\begin{equation}\label{hkdomt}
 g\oplus g^\prime \in \dom T\qquad \text{and}\quad T(g\oplus g^\prime)=- l \oplus l^\prime.
\end{equation}
From Green's identity we then obtain
\begin{equation*}
\begin{split}
&\bigl(T(f\oplus f^\prime),g\oplus g^\prime\bigr)-\bigl(f \oplus f^\prime, T( g \oplus g^\prime)\bigr)\\
&\qquad=(\cL_\Omega f,g)_\Omega-(f,\cL_\Omega g)_\Omega+(\cL_{\Omega^\prime} f^\prime,g^\prime)_{\Omega^\prime}
-(f^\prime,\cL_{\Omega^\prime}g^\prime)_{\Omega^\prime}\\
&\qquad=\left(f|_\cC,\frac{\partial g}{\partial\nu}\Bigl|_\cC\right)_\cC-
\left(\frac{\partial f}{\partial\nu}\Bigl|_\cC,g|_\cC\right)_\cC
+\left(f^\prime|_\cC,\frac{\partial g^\prime}{\partial\nu^\prime}\Bigl|_\cC\right)_\cC
-\left(\frac{\partial f^\prime}{\partial\nu^\prime}\Bigl|_\cC,g^\prime|_\cC\right)_\cC\\
&\qquad=\left(f|_\cC,\frac{\partial g}{\partial\nu}\Bigl|_\cC+\frac{\partial g^\prime}{\partial\nu^\prime}
\Bigl|_\cC\right)_\cC-
\left(\frac{\partial f}{\partial\nu}\Bigl|_\cC+\frac{\partial f^\prime}{\partial\nu^\prime}\Bigl|_\cC,g|_\cC\right)_\cC,
\end{split}
\end{equation*}
where we have used that $f\oplus f^\prime,\,g\oplus g^\prime\in\dom T$ satisfy $f|_\cC=f^\prime|_\cC$ and 
$g|_\cC=g^\prime|_\cC$. Inserting \eqref{hkdomt} in \eqref{polk} and comparing this with
the above relation shows that the identity
\begin{equation}\label{compare}
\left(f|_\cC,\frac{\partial g}{\partial\nu}\Bigl|_\cC+\frac{\partial g^\prime}{\partial\nu^\prime}\Bigl|_\cC+\,\varphi\right)_\cC
=\left(\frac{\partial f}{\partial\nu}\Bigl|_\cC+\frac{\partial f^\prime}{\partial\nu^\prime}\Bigl|_\cC,g|_\cC-\psi\right)_\cC
\end{equation}
holds for all $f\oplus f^\prime\in \dom T$. As the mapping \eqref{tracemapt} is surjective 
and $H^{3/2}(\cC) \times H^{1/2}(\cC)$ is dense in $L^2(\cC)\oplus L^2(\cC)$ we conclude from \eqref{compare} that
\begin{equation*}
\varphi=-\left(\frac{\partial g}{\partial\nu}\Bigl|_\cC+\frac{\partial g^\prime}{\partial\nu^\prime}\Bigl|_\cC\right)
\qquad\text{and}\qquad
\psi=g|_\cC
\end{equation*}
holds. Hence we have seen that the element \eqref{mort} in $\cM^\bot$ is of the form
\begin{equation}\label{mort2}
\left[\begin{matrix} - T(g\oplus g^\prime) \\ g\oplus g^\prime \\  -\tfrac{\partial g}{\partial \nu}|_\cC - 
\tfrac{\partial g^\prime}{\partial \nu^\prime}|_\cC \\ g\vert_\cC \end{matrix}\right]
\end{equation}
for some $g\oplus g^\prime\in \dom T$.
It is not difficult to check that conversely an element as in \eqref{mort2} belongs to $\cM^\bot$. Therefore the
orthogonal complement of $\cM$ is given by
\begin{equation*}
\cM^\bot=\left\{\left[\begin{matrix} - T(g\oplus g^\prime) \\ g\oplus g^\prime \\  
-\tfrac{\partial g}{\partial n}\bigl|_\cC -
\tfrac{\partial g^\prime}{\partial \nu^\prime}\bigl|_\cC \\ g\vert_\cC \end{matrix}\right]: 
 g\oplus g^\prime \in \dom T  \right\}
\subset L^2(\dR^n)\oplus L^2(\dR^n)\oplus L^2(\cC)\oplus L^2(\cC)
\end{equation*}
and together with \eqref{nbot} we find that the sum of $\cM^\bot$ and $\cN^\bot$ is
\begin{equation*}
\cM^\bot +\cN^\bot =\left\{\left[\begin{matrix} - T(g\oplus g^\prime)\\ g\oplus g^\prime \end{matrix}\right]:g\oplus g^\prime\in\dom T 
\right\}\oplus L^2(\cC)\oplus L^2(\cC).
\end{equation*}
The assumption that $T$ is closed implies that $\cM^\bot +\cN^\bot$ is a closed subspace of $L^2(\dR^n)\oplus L^2(\dR^n)\oplus L^2(\cC)\oplus L^2(\cC)$.
But then according to \cite[IV Theorem 4.8]{K76} also $\cM+\cN$ is a closed subspace 
of $L^2(\dR^n)\oplus L^2(\dR^n)\oplus L^2(\cC)\oplus L^2(\cC)$ which is a contradiction to \eqref{cmcn}. Thus $T$ can not be closed.
\end{proof}

The following lemma will be useful later in this section.

\begin{lemma}\label{usefullemma}
Let $S$ and $T$ be as in Theorem~\ref{opscoup} and let $\widetilde A$ be the selfadjoint realization of $\cL$ in $L^2(\dR^n)$
defined on $H^2(\dR^n)$. For a function $f\oplus f^\prime\in \dom T$, where $f\in H^2(\Omega)$ and $f^\prime\in H^2(\Omega^\prime)$,
we have
\begin{equation*}
f\oplus f^\prime\in \dom \widetilde A\qquad\text{if and only if}\qquad \frac{\partial f}{\partial\nu}\Bigl|_\cC=
-\frac{\partial f^\prime}{\partial\nu^\prime}\Bigl|_\cC.
\end{equation*}
\end{lemma}

\begin{proof}
For a function $f\oplus f^\prime\in \dom\widetilde A=H^2(\dR^n)$ it is clear that $\tfrac{\partial f}{\partial\nu}|_\cC=
-\tfrac{\partial f^\prime}{\partial\nu^\prime}|_\cC$ holds. Conversely, let $f\oplus f^\prime\in\dom T$ and assume 
\begin{equation}
 \frac{\partial f}{\partial\nu}\Bigl|_\cC=
-\frac{\partial f^\prime}{\partial\nu^\prime}\Bigl|_\cC.
\end{equation}
Then also $f|_\cC=f^\prime|_\cC$ and since every $g\oplus g^\prime\in\dom \widetilde A$ satisfies
\begin{equation*}
g|_\cC=g^\prime|_\cC\qquad\text{and}\qquad \frac{\partial g}{\partial\nu}\Bigl|_\cC=
-\frac{\partial g^\prime}{\partial\nu^\prime}\Bigl|_\cC
\end{equation*}
Green's identity implies
\begin{equation*}
\begin{split}
&\qquad\qquad\bigl(\widetilde A (g\oplus g^\prime),f\oplus f^\prime\bigr)-\bigl(g\oplus g^\prime,T(f\oplus f^\prime)\bigr)\\
&=\left(g|_\cC,\frac{\partial f}{\partial\nu}\Bigl|_\cC\right)_\cC-
\left(\frac{\partial g}{\partial\nu}\Bigl|_\cC,f|_\cC\right)_\cC+
\left(g^\prime|_\cC,\frac{\partial f^\prime}{\partial\nu}\Bigl|_\cC\right)_\cC-
\left(\frac{\partial g^\prime}{\partial\nu}\Bigl|_\cC,f^\prime|_\cC\right)_\cC=0.
\end{split}
\end{equation*}
Therefore $f\oplus f^\prime\in\dom \widetilde A^*=\dom\widetilde A$.
\end{proof}

Next we define a mapping $\Gamma_{\lambda_0}$ which satisfies the assumptions in the definition of a generalized $Q$-function.
For this let $A$ be the selfadjoint operator
in $L^2(\dR^n)$ in \eqref{aschro} and \eqref{aschroe} which is the orthogonal sum of the Dirichlet operators $A_\Omega$ and $A_{\Omega^\prime}$ in $L^2(\Omega)$ and $L^2(\Omega^\prime)$, respectively. For $\lambda\in\rho(A)$ the domain of the operator $T$ in Theorem~\ref{opscoup} can be decomposed
in
\begin{equation}\label{deco3}
\begin{split}
\dom T&=\dom A\,\dot+\,\cN_\lambda(T)\\
&=\bigl\{f\oplus f^\prime\in H^2(\Omega)\oplus H^2(\Omega^\prime): f|_\cC=f^\prime|_\cC=0\bigr\}\, \dot+\, \cN_\lambda(T),
\end{split}
\end{equation}
cf. \eqref{decoall}.
Let us fix some $\lambda_0\in\rho(A)$. The decomposition \eqref{deco3} 
and the surjectivity of the map 
\begin{equation}\label{tracemaptt}
\dom T\ni f \oplus f^\prime\,\,\mapsto\,\,\left\{f|_\cC,\frac{\partial f}{\partial\nu}\Bigl|_\cC+\frac{\partial f^\prime}
{\partial\nu^\prime}\Bigl|_\cC\right\}\in H^{3/2}(\cC)\,\times\, H^{1/2}(\cC),
\end{equation}
cf. \eqref{tracemap}, \eqref{tracemapt} imply that for a given function $\varphi\in H^{3/2}(\cC)$ there exists a unique function 
$f_{\lambda_0}\oplus f^\prime_{\lambda_0}\in \cN_{\lambda_0}(T)$ such that $f_{\lambda_0}|_\cC=f^\prime_{\lambda_0}|_\cC=\varphi$. 
Let $\Gamma_{\lambda_0}$
be the mapping that assigns $f_{\lambda_0}\oplus f^\prime_{\lambda_0}$ to $\varphi$,
\begin{equation}\label{gammalambda0}
H^{3/2}(\cC)\ni \varphi\mapsto \Gamma_{\lambda_0}\varphi:=f_{\lambda_0}\oplus f^\prime_{\lambda_0}.
\end{equation}
Similarly as in the previous section $\Gamma_{\lambda_0}$ will be regarded as an operator from $L^2(\cC)$ to $L^2(\dR^n)$ with 
$\dom\Gamma_{\lambda_0}=H^{3/2}(\cC)$ and $\ran\Gamma_{\lambda_0}=\cN_{\lambda_0}(T)$. Observe that the function $\Gamma_{\lambda_0}\varphi=f_{\lambda_0}\oplus f^\prime_{\lambda_0}$ consists of an $H^2(\Omega)$-solution $f_{\lambda_0}$ of $\cL_\Omega u=\lambda_0 u$
and an $H^2(\Omega^\prime)$-solution $f^\prime_{\lambda_0}$ of $\cL_{\Omega^\prime} u^\prime=\lambda_0 u^\prime$ satisfying the boundary conditions $\varphi=f_{\lambda_0}|_\cC=f^\prime_{\lambda_0}|_\cC$.

The following proposition parallels Proposition~\ref{Gammalambda0prop}.

\begin{proposition}\label{Gammalambda0prop2}
Let $\lambda_0\in\rho(A)$, let $\Gamma_{\lambda_0}$ be as in \eqref{gammalambda0} and let $\lambda\in\rho(A)$. 
Then the following holds:
\begin{enumerate}
\item [{\rm (i)}] $\Gamma_{\lambda_0}$ is a bounded operator from $L^2(\cC)$ in $L^2(\dR^n)$ with dense
domain $H^{3/2}(\cC)$;
\item [{\rm (ii)}] The operator $\Gamma(\lambda)=(I+(\lambda-\lambda_0)(A-\lambda)^{-1})\Gamma_{\lambda_0}$ is given
by 
\begin{equation*}
\Gamma(\lambda) \varphi =f_\lambda\oplus f^\prime_\lambda ,\quad\text{where}\quad f_\lambda\oplus f^\prime_\lambda\in\cN_\lambda(T)\,\,\,\,\text{and}\,\,\,\,
f_\lambda\vert_\cC=\varphi=f^\prime_\lambda\vert_\cC;
\end{equation*}
\item [{\rm (iii)}] The mapping $\Gamma(\bar\lambda)^*:L^2(\dR^n)\rightarrow L^2(\cC)$ satisfies 
\begin{equation*}
\Gamma(\bar\lambda)^*(A-\lambda)h=-\frac{\partial f}{\partial\nu}\Bigl|_\cC - \frac{\partial f^\prime}{\partial\nu^\prime}\Bigl|_\cC,\qquad 
h=f\oplus f^\prime\in\dom A.
\end{equation*}
\end{enumerate}
\end{proposition}

\begin{proof}
We start with the proof (ii). Let $\varphi\in H^{3/2}(\cC)$ and choose the unique elements $f_\lambda\oplus f^\prime_\lambda\in\cN_\lambda(T)$ and
$f_{\lambda_0}\oplus f^\prime_{\lambda_0}\in\cN_{\lambda_0}(T)$ such that
\begin{equation*}
f_\lambda|_\cC=f^\prime_\lambda|_\cC=\varphi=f_{\lambda_0}|_\cC=f^\prime_{\lambda_0}|_\cC
\end{equation*}
holds. By definition $\Gamma_{\lambda_0}\varphi=f_{\lambda_0}\oplus f^\prime_{\lambda_0}$ and therefore
\begin{equation*}
\begin{split}
 \Gamma(\lambda)\varphi&=\Gamma_{\lambda_0}\varphi+(\lambda-\lambda_0)(A-\lambda)^{-1}\Gamma_{\lambda_0}\varphi\\
&=f_{\lambda_0}\oplus f^\prime_{\lambda_0}+(\lambda-\lambda_0)(A-\lambda)^{-1}\Gamma_{\lambda_0}\varphi.
\end{split}
\end{equation*}
Since $(\lambda-\lambda_0)(A-\lambda)^{-1}\Gamma_{\lambda_0}\varphi$ is a function belonging to $\dom A$ we have
\begin{equation*}
\bigl((\lambda-\lambda_0)(A-\lambda)^{-1}\Gamma_{\lambda_0}\varphi\bigr)\bigl|_\cC=0,
\end{equation*}
cf. \eqref{aschroe}. This implies
\begin{equation*}
(\Gamma(\lambda)\varphi)|_\cC= (\Gamma_{\lambda_0}\varphi)|_\cC=\bigl(f_{\lambda_0}\oplus f^\prime_{\lambda_0}\bigr)|_\cC=f_{\lambda_0}|_\cC=f^\prime_{\lambda_0}|_\cC=\varphi
\end{equation*}
and since $\ran\Gamma(\lambda)=\cN_\lambda(T)$, see Lemma~\ref{gamlem}, and $f_\lambda\oplus f_\lambda^\prime$ is the unique function in $\cN_\lambda(T)$ with
$f_\lambda|_\cC=f^\prime_\lambda|_\cC=\varphi$ we conclude $\Gamma(\lambda)\varphi=f_\lambda\oplus f_\lambda^\prime$.

\vskip 0.3cm\noindent
Next we verify (iii). Observe that then $\Gamma(\bar\lambda)^*$, $\lambda\in\rho(A)$, 
is a closed operator which is defined on the whole space, i.e.,
$\Gamma(\bar\lambda)^*$ is bounded and hence assertion (i) follows by setting $\lambda_0=\bar\lambda$.
Let $\varphi\in H^{3/2}(\cC)$ and choose the unique function $f_{\bar\lambda}\oplus f^\prime_{\bar\lambda}\in\cN_{\bar\lambda}(T)$ 
such that 
\begin{equation}\label{asdf}
f_{\bar\lambda}\vert_\cC=f^\prime_{\bar\lambda}\vert_\cC=\varphi
\end{equation}
holds. Then $\Gamma(\bar\lambda)\varphi=f_{\bar\lambda}\oplus f^\prime_{\bar\lambda}$ and for each 
$h=f\oplus f^\prime\in\dom A$, where $f\in H^2(\Omega)$, $f^\prime\in H^2(\Omega^\prime)$, we have
\begin{equation*}
\begin{split}
\bigl(\Gamma(\bar\lambda)\varphi,(A-\lambda)h\bigr)&=\bigl(f_{\bar\lambda}\oplus f^\prime_{\bar\lambda}, A(f\oplus f^\prime)\bigr)
-\bigl(T(f_{\bar\lambda}\oplus f^\prime_{\bar\lambda}),f\oplus f^\prime\bigr)\\
&=(f_{\bar\lambda},\cL_\Omega f)_\Omega-(\cL_\Omega f_{\bar\lambda},f)_\Omega+
(f^\prime_{\bar\lambda},\cL_{\Omega^\prime} f^\prime)_{\Omega^\prime}-(\cL_{\Omega^\prime} f^\prime_{\bar\lambda},f^\prime)_{\Omega^\prime}.
\end{split}
\end{equation*}
With the help of Green's identity this can be rewritten as
\begin{equation*}
\left(\frac{\partial f_{\bar\lambda}}{\partial\nu}\Bigl|_\cC,f|_\cC\right)_\cC-
\left(f_{\bar\lambda}|_\cC,\frac{\partial f}{\partial\nu}\Bigl|_\cC\right)_\cC
+\left(\frac{\partial f^\prime_{\bar\lambda}}{\partial\nu^\prime}\Bigl|_\cC,f^\prime|_\cC\right)_\cC-
\left(f^\prime_{\bar\lambda}|_\cC,\frac{\partial f^\prime}{\partial\nu^\prime}\Bigl|_\cC\right)_\cC.
\end{equation*}
Since for $h=f\oplus f^\prime\in\dom A$ we have $f|_\cC=f^\prime|_\cC=0$ we conclude from the above calculation and \eqref{asdf} that
\begin{equation*}
\bigl(\Gamma(\bar\lambda)\varphi,(A-\lambda)h\bigr)=-
\left(\varphi,\frac{\partial f}{\partial\nu}\Bigl|_\cC+\frac{\partial f^\prime}{\partial\nu^\prime}\Bigl|_\cC\right)_\cC
\end{equation*}
holds for every $\varphi\in H^{3/2}(\cC)=\dom\Gamma(\bar\lambda)$. Hence $(A-\lambda)h\in\dom\Gamma(\bar\lambda)^*$
and
\begin{equation*}
\Gamma(\bar\lambda)^*(A-\lambda)h=-\frac{\partial f}{\partial\nu}\Bigl|_\cC-\frac{\partial f^\prime}{\partial\nu^\prime}\Bigl|_\cC,
\qquad h=f\oplus f^\prime\in\dom A.
\end{equation*}
Furthermore, for $\lambda\in\rho(A)$ we have $\ran(A-\lambda)=L^2(\dR^n)$, so that $\Gamma(\bar\lambda)^*$ is a bounded operator defined on $L^2(\dR^n)$.
\end{proof}

Next we define a function $Q$ in a similar way as the Dirichlet-to-Neumann map in Definition~\ref{dirneu}. For this we make use of 
the decomposition \eqref{deco3}. Namely, for $\lambda\in\rho(A)$ and $\varphi\in H^{3/2}(\cC)$ there exists a unique function
$f_\lambda\oplus f^\prime_\lambda\in\cN_\lambda(T)$ such that $f_\lambda\vert_\cC=f^\prime_\lambda\vert_\cC=\varphi$. The operator 
$Q(\lambda)$ in $L^2(\cC)$ is now defined by
\begin{equation}\label{qcoup}
Q(\lambda)\varphi:=-\frac{\partial f_\lambda}{\partial\nu}\Bigl|_\cC-
\frac{\partial f_\lambda^\prime}{\partial\nu^\prime}\Bigl|_\cC,
\qquad\varphi\in \dom Q(\lambda)=H^{3/2}(\cC).
\end{equation}
Observe that $\ran Q(\lambda)\subset H^{1/2}(\cC)$ holds. Roughly speaking, up to a minus sign $Q(\lambda)$ maps the Dirichlet boundary value of 
the $H^2$-solutions of $\cL_\Omega u=\lambda u$ and $\cL_{\Omega^\prime}u^\prime=\lambda u^\prime$, $u|_\cC=u^\prime|_\cC$,
onto the sum of the Neumann boundary values of these solutions. We mention that in the analysis of so-called intermediate
Hamiltonians a modified form of such a Dirichlet-to-Neumann map has been used in \cite{MPP07}.

In the following theorem it turns out that $Q$ can be interpreted as a generalized $Q$-function  and 
the difference of the resolvents of $A$ and $\widetilde A$ is expressed with the help of $Q$.

\begin{theorem}\label{bigthm2}
Let $\cL$ be the elliptic differential expression in \eqref{cl2} and let $A$ and $\widetilde A$ be the selfadjoint
realizations of $\cL$ in \eqref{aschro}-\eqref{aschroe} and \eqref{atildeschroe}, respectively. Let $S$ and $T$
be the operators in Theorem~\ref{opscoup}, define
$\Gamma(\lambda)$ as in 
Proposition~\ref{Gammalambda0prop2} and let $Q(\lambda)$, $\lambda\in\rho(A)$, be as in \eqref{qcoup}.
Then the following holds:
\begin{enumerate}
\item [{\rm (i)}] $Q$ is a generalized $Q$-function of the triple $\{S,A,T\}$; 
\item [{\rm (ii)}] The operator $Q(\lambda)$ is injective 
for all $\lambda\in\rho(A)\cap\rho(\widetilde A)$ and the resolvent formula
\begin{equation}\label{resform2}
(A-\lambda)^{-1}-(\widetilde A-\lambda)^{-1}=\Gamma(\lambda) Q(\lambda)^{-1} \Gamma(\bar\lambda)^*
\end{equation}
holds;
\item [{\rm (iii)}] For $p\in\dN$ and $2p+1>n$ the difference of the resolvents in \eqref{resform2}
belongs to the von Neumann-Schatten class $\sS_p(L^2(\Omega))$.
\end{enumerate}
\end{theorem}

\begin{proof} 
Let us prove assertion (i). Before the defining relation \eqref{q} for a generalized $Q$-function will be verified we show that 
the operator $Q(\mu)^*$ is an extension of $Q(\bar\mu)$, $\mu\in\rho(A)$. For this let $\psi\in H^{3/2}(\cC)$ and choose
the unique element $f_{\bar\mu}\oplus f^\prime_{\bar\mu}\in\cN_{\bar\mu}(T)$ with the property 
$f_{\bar\mu}|_\cC=f^\prime_{\bar\mu}|_\cC=\psi$. For $\varphi\in H^{3/2}(\cC)$ let $f_{\mu}\oplus f^\prime_{\mu}\in\cN_{\mu}(T)$
be such that $f_{\mu}|_\cC=f^\prime_{\mu}|_\cC=\varphi$ holds. By the definition of $Q$ in \eqref{qcoup} we have
\begin{equation*}
Q(\mu)\varphi=-\frac{\partial f_\mu}{\partial\nu}\Bigl|_\cC-\frac{\partial f^\prime_\mu}{\partial\nu^\prime}\Bigl|_\cC
\quad\text{and}\quad
Q(\bar\mu)\psi=-\frac{\partial f_{\bar\mu}}{\partial\nu}\Bigl|_\cC-\frac{\partial f^\prime_{\bar\mu}}{\partial\nu^\prime}\Bigl|_\cC.
\end{equation*}
This gives
\begin{equation}\label{qadj}
(Q(\mu)\varphi,\psi)=-\left(\frac{\partial f_\mu}{\partial\nu}\Bigl|_\cC,f_{\bar\mu}|_\cC\right)_\cC
-\left(\frac{\partial f^\prime_\mu}{\partial\nu^\prime}\Bigl|_\cC,f^\prime_{\bar\mu}|_\cC\right)_\cC
\end{equation}
and since
\begin{equation*}
 \begin{split}
\left(f_\mu|_\cC,\frac{\partial f_{\bar\mu}}{\partial\nu}\Bigl|_\cC\right)_\cC- 
\left(\frac{\partial f_\mu}{\partial\nu}\Bigl|_\cC,f_{\bar\mu}|_\cC\right)_\cC&=
(\cL_\Omega f_\mu,f_{\bar\mu})_\Omega-(f_\mu,\cL_\Omega f_{\bar\mu})_\Omega=0,\\
\left(f^\prime_\mu|_\cC,\frac{\partial f^\prime_{\bar\mu}}{\partial\nu^\prime}\Bigl|_\cC\right)_\cC-
\left(\frac{\partial f^\prime_\mu}{\partial\nu^\prime}\Bigl|_\cC,f^\prime_{\bar\mu}|_\cC\right)_\cC
&=(\cL_{\Omega^\prime} f^\prime_\mu,f^\prime_{\bar\mu})_{\Omega^\prime}-(f^\prime_\mu,\cL_{\Omega^\prime} f^\prime_{\bar\mu})_{\Omega^\prime}=0
 \end{split}
\end{equation*}
we can rewrite \eqref{qadj} in the form
\begin{equation*}
 (Q(\mu)\varphi,\psi)=-\left(f_\mu|_\cC,\frac{\partial f_{\bar\mu}}{\partial\nu}\Bigl|_\cC\right)_\cC-
\left(f^\prime_\mu|_\cC,\frac{\partial f^\prime_{\bar\mu}}{\partial\nu^\prime}\Bigl|_\cC\right)_\cC
=-\left(\varphi,\frac{\partial f_{\bar\mu}}{\partial\nu}\Bigl|_\cC+\frac{\partial f^\prime_{\bar\mu}}{\partial\nu^\prime}\Bigl|_\cC\right)_\cC.
\end{equation*}
This is true for every $\varphi\in\dom Q(\mu)$ and hence we conclude $\psi\in\dom Q(\mu)^*$ and
\begin{equation*}
Q(\mu)^*\psi=-\frac{\partial f_{\bar\mu}}{\partial\nu}\Bigl|_\cC-\frac{\partial f^\prime_{\bar\mu}}{\partial\nu^\prime}\Bigl|_\cC
=Q(\bar\mu)\psi.
\end{equation*}

Let $\Gamma(\cdot)$ be as in Proposition~\ref{Gammalambda0prop2}. We prove now that 
\begin{equation}\label{qrel2}
Q(\lambda)-Q(\mu)^*=(\lambda-\bar\mu)\Gamma(\mu)^*\Gamma(\lambda),\qquad \lambda,\mu\in\rho(A)
\end{equation}
holds on $\dom\Gamma(\lambda)=H^{3/2}(\cC)$. For this let $\varphi,\psi\in H^{3/2}(\cC)$ and choose the unique elements 
$f_\lambda\oplus f^\prime_\lambda\in\cN_\lambda(T)$, $f_\mu\oplus f^\prime_\mu\in\cN_\mu(T)$ with the properties
\begin{equation}\label{bts}
f_\lambda|_\cC=f^\prime_\lambda|_\cC=\varphi\quad\text{and}\quad f_\mu|_\cC=f^\prime_\mu|_\cC=\psi.
\end{equation}
Then according to Proposition~\ref{Gammalambda0prop2}~(ii) $\Gamma(\lambda)\varphi=f_\lambda\oplus f^\prime_\lambda$ and $\Gamma(\mu)\psi=f_\mu\oplus f^\prime_\mu$ and
by the definition of $Q(\cdot)$ in \eqref{qcoup} we have
\begin{equation*}
Q(\lambda)\varphi=-\frac{\partial f_\lambda}{\partial\nu}\Bigl|_\cC - \frac{\partial f^\prime_\lambda}{\partial\nu^\prime}\Bigl|_\cC
\quad\text{and}\quad
Q(\mu)\psi=-\frac{\partial f_\mu}{\partial\nu}\Bigl|_\cC - \frac{\partial f^\prime_\mu}{\partial\nu^\prime}\Bigl|_\cC.
\end{equation*}
Therefore 
\begin{equation*}
\bigl((Q(\lambda)-Q(\mu)^*)\varphi,\psi\bigr)_\cC=-\left(\frac{\partial f_\lambda}{\partial\nu}\Bigl|_\cC + 
\frac{\partial f^\prime_\lambda}{\partial\nu^\prime}\Bigl|_\cC,\psi \right)_\cC+
\left(\varphi,\frac{\partial f_\mu}{\partial\nu}\Bigl|_\cC + 
\frac{\partial f^\prime_\mu}{\partial\nu^\prime}\Bigl|_\cC \right)_\cC
\end{equation*}
and inserting \eqref{bts} gives
\begin{equation*}
-\left(\frac{\partial f_\lambda}{\partial\nu}\Bigl|_\cC,f_\mu|_\cC\right)_\cC -
\left(\frac{\partial f^\prime_\lambda}{\partial\nu^\prime}\Bigl|_\cC,f^\prime_\mu|_\cC \right)_\cC+
\left(f_\lambda|_\cC,\frac{\partial f_\mu}{\partial\nu}\Bigl|_\cC\right)_\cC + 
\left(f_\lambda^\prime|_\cC,\frac{\partial f^\prime_\mu}{\partial\nu^\prime}\Bigl|_\cC \right)_\cC.
\end{equation*}
Making use of Green's identity the above relations then become
\begin{equation*}
\begin{split}
&\bigl((Q(\lambda)-Q(\mu)^*)\varphi,\psi\bigr)_\cC\\
&\qquad\quad=(\cL_\Omega f_\lambda,f_\mu)_\Omega-(f_\lambda,\cL_\Omega f_\mu)_\Omega 
+(\cL_{\Omega^\prime}f^\prime_\lambda,f^\prime_\mu)_{\Omega^\prime}-(f^\prime_\lambda,\cL_{\Omega^\prime}f^\prime_\mu)_{\Omega^\prime}\\
&\qquad\quad=(\lambda-\bar\mu)\bigl( (f_\lambda,f_\mu)_\Omega+(f^\prime_\lambda,f^\prime_\mu)_{\Omega^\prime} \bigr)
=(\lambda-\bar\mu)\bigl(f_\lambda\oplus f^\prime_\lambda,f_\mu\oplus f^\prime_\mu\bigr)\\
&\qquad\quad=(\lambda-\bar\mu)(\Gamma(\lambda)\varphi,\Gamma(\mu)\psi)=\bigl((\lambda-\bar\mu)\Gamma(\mu)^*\Gamma(\lambda)\varphi,\psi\bigr)_\cC.
\end{split}
\end{equation*}
Since this is true for any $\psi\in H^{3/2}(\cC)$ we conclude that \eqref{qrel2} holds on $H^{3/2}(\cC)$. Thus $Q$ in \eqref{qcoup}
is a generalized $Q$-function for the triple $\{S,A,T\}$.

\vskip 0.3cm\noindent
(ii) We check first that $\ker Q(\lambda)=\{0\}$ holds for $\lambda\in\rho(A)\cap\rho(\widetilde A)$. Assume that
$Q(\lambda)\varphi=0$ for some $\varphi\in H^{3/2}(\cC)$ and let $f_\lambda\oplus f_\lambda^\prime\in\cN_\lambda(T)$
be the unique element with the property $f_\lambda|_\cC=f^\prime_\lambda|_\cC=\varphi$. Then the definition of $Q$ and the assumption
$Q(\lambda)\varphi=0$ imply
\begin{equation*}
\frac{\partial f_\lambda}{\partial\nu}\Bigl|_\cC=-\frac{\partial f^\prime_\lambda}{\partial\nu^\prime}\Bigl|_\cC.
\end{equation*}
According to Lemma~\ref{usefullemma} this yields $f_\lambda\oplus f_\lambda^\prime\in\dom\widetilde A\cap\cN_\lambda(T)$. But as $\lambda\in\rho(\widetilde A)$ we conclude $f_\lambda=0$ and $f^\prime_\lambda=0$, and hence $\varphi=0$.

Now we prove the formula \eqref{resform2} for the difference of the resolvents of $A$ and $\widetilde A$. By the above argument $Q(\lambda)^{-1}$
exists for $\lambda\in\rho(A)\cap\rho(\widetilde A)$. Furthermore, \eqref{tracemaptt} implies $\ran Q(\lambda)=H^{1/2}(\cC)$ and it 
follows from Proposition~\ref{Gammalambda0prop2} that the right hand side in \eqref{resform2} is well defined.

Let $h\in L^2(\dR^n)$ and define the function $k$ as
\begin{equation}\label{k}
k=(A-\lambda)^{-1}h-\Gamma(\lambda)Q(\lambda)^{-1}\Gamma(\bar\lambda)^*h.
\end{equation}
We show $k\in\dom\widetilde A$. First of all it is clear that $k\in\dom T$ since $(A-\lambda)^{-1}h\in\dom A\subset\dom T$ and
$\Gamma(\lambda)$ maps into $\cN_\lambda(T)$. Therefore $k=g\oplus g^\prime$, where $g\in H^2(\Omega)$, $g^\prime\in H^2(\Omega^\prime)$, and 
$g|_\cC=g^\prime|_\cC$. According to Lemma~\ref{usefullemma} for $k\in\dom \widetilde A$ it is sufficient to check
\begin{equation}\label{gbc}
\frac{\partial g}{\partial\nu}\Bigl|_\cC+\frac{\partial g^\prime}{\partial\nu^\prime}\Bigl|_\cC=0.
\end{equation}
We proceed in a similar way as in the proof of Theorem~\ref{bigthm1}. Let $h_A=f_A\oplus f_A^\prime\in\dom A$ be such that
$h=(A-\lambda)h_A$. Making use of Proposition~\ref{Gammalambda0prop2}~(iii) we obtain
\begin{equation}\label{k2}
 k=h_A+\Gamma(\lambda)Q(\lambda)^{-1}\left(\frac{\partial f_A}{\partial\nu}\Bigl|_\cC+\frac{\partial f_A^\prime}{\partial\nu^\prime}\Bigl|_\cC\right)
\end{equation}
from \eqref{k}. Let 
\begin{equation*}
\cN_\lambda(T)\ni f_\lambda\oplus f_\lambda^\prime:=\Gamma(\lambda)Q(\lambda)^{-1}\left(\frac{\partial f_A}{\partial\nu}\Bigl|_\cC+\frac{\partial f_A^\prime}{\partial\nu^\prime}\Bigl|_\cC\right).
\end{equation*}
Then by Proposition~\ref{Gammalambda0prop2}~(ii) we have
\begin{equation*}
f_\lambda|_\cC=f^\prime_\lambda|_\cC=Q(\lambda)^{-1}\left(\frac{\partial f_A}{\partial\nu}\Bigl|_\cC+\frac{\partial f_A^\prime}{\partial\nu^\prime}\Bigl|_\cC\right).
\end{equation*}
This together with the definition of $Q(\lambda)$ in \eqref{qcoup} implies
\begin{equation*}
 \frac{\partial f_A}{\partial\nu}\Bigl|_\cC+\frac{\partial f_A^\prime}{\partial\nu^\prime}\Bigl|_\cC=
Q(\lambda)(f_\lambda|_\cC)=Q(\lambda)(f^\prime_\lambda|_\cC)
= -\frac{\partial f_\lambda}{\partial\nu}\Bigl|_\cC-\frac{\partial f_\lambda^\prime}{\partial\nu^\prime}\Bigl|_\cC.
\end{equation*}
Hence we conclude that the function $k=g\oplus g^\prime$ in \eqref{k2} fulfils \eqref{gbc}, i.e., $k\in\dom\widetilde A$.
From \eqref{k} and $A,\widetilde A\subset T$ we obtain
\begin{equation*}
 (\widetilde A-\lambda)k=(T-\lambda)(A-\lambda)^{-1}h-(T-\lambda)\Gamma(\lambda)Q(\lambda)^{-1}\Gamma(\bar\lambda)^*h=h
\end{equation*}
and now $k=(\widetilde A-\lambda)^{-1}h$ and \eqref{k} imply \eqref{resform2}.
\end{proof}

The following corollaries can be proved in the same way as Corollary~\ref{prop1} and Corollary~\ref{prop2}.

\begin{corollary}\label{prop12}
For $\lambda,\lambda_0\in\rho(A)$ the following holds.
\begin{enumerate}
 \item [{\rm (i)}] $Q(\lambda)$ is a non-closed unbounded operator in $L^2(\cC)$ defined
                   on $H^{3/2}(\cC)$ with $\ran Q(\lambda)\subset H^{1/2}(\cC)$;
                   
 \item [{\rm (ii)}] $Q(\lambda)-\RE Q(\lambda_0)$ is a non-closed bounded operator in $L^2(\cC)$ defined
                   on $H^{3/2}(\cC)$;
 \item [{\rm (iii)}] the closure $\widetilde Q(\lambda)$ of the operator $Q(\lambda)-\RE Q(\lambda_0)$ in 
$L^2(\cC)$ satisfies 
$$\frac{d}{d\lambda}\,\widetilde Q(\lambda)=\Gamma(\bar\lambda)^*\overline{\Gamma(\lambda)}$$
and $\widetilde Q$ is a $\cL(L^2(\cC))$-valued
 Nevanlinna function. 
 \end{enumerate}
\end{corollary}

\begin{corollary}\label{prop22}
For $\lambda\in\rho(A)\cap\rho(\widetilde A)$ the following holds.
\begin{enumerate}
 \item [{\rm (i)}] $Q(\lambda)^{-1}$ is a non-closed bounded operator in $L^2(\cC)$ defined
                    on $H^{1/2}(\cC)$ with $\ran Q(\lambda)^{-1}=H^{3/2}(\cC)$;
 \item [{\rm (ii)}] the closure $\overline{Q(\lambda)^{-1}}$ is a compact operator in $L^2(\cC)$;
 \item [{\rm (iii)}] the function $\lambda\mapsto -\overline{Q(\lambda)^{-1}}$ is a $\cL(L^2(\cC))$-valued
 Nevanlinna function.
 \end{enumerate}
\end{corollary}

As a corollary of Theorem~\ref{bigthm2} we obtain a trace formula for the difference of the resolvents of $A$
and $\widetilde A$.

\begin{corollary}
Let the assumptions be as in Theorem~\ref{bigthm2}, let $\widetilde Q$ be the Nevanlinna function
from Corollary~\ref{prop12} and suppose, in addition, $n=2$. Then 
\begin{equation*}
\tr\bigl((A-\lambda)^{-1}-(\widetilde A-\lambda)^{-1}\bigr)=\tr\left(\overline{Q(\lambda)^{-1}}
\frac{d}{d\lambda}\,\widetilde Q(\lambda)\right)
\end{equation*}
holds for all $\lambda\in\rho(A)\cap\rho(\widetilde A)$. 
\end{corollary}

\end{document}